%% file: Arxive_paper1_v3.tex
\documentclass{amsart} \topmargin -15mm \textwidth 160 true mm
\usepackage{t1enc}
\usepackage{graphicx}
\usepackage{float}

\usepackage{caption}
\usepackage{subcaption}
\usepackage{epstopdf}
\usepackage{mathrsfs}
\usepackage{multicol}
\usepackage{setspace}
\usepackage{amsmath, amssymb, amsfonts, amstext, amsthm,amscd}
\usepackage{relsize}
\usepackage[mathscr]{euscript}
\usepackage{enumerate, mathtools}
\usepackage{chngcntr}
\usepackage{tikz-cd}
\usepackage{tikz,color,soul, verbatim}
\usepackage{soul}
\usetikzlibrary{arrows.meta,arrows}
\usepackage[english]{babel}
\usepackage[pdftex]{hyperref}
\usepackage[foot]{amsaddr}
\newtheorem{theorem}{Theorem}[section]
\newtheorem{proposition}[theorem]{Proposition}
\newtheorem{lemma}[theorem]{Lemma}

\newtheorem{corollary}[theorem]{Corollary}
\theoremstyle{definition}
\newtheorem{definition}[theorem]{Definition}
\newtheorem{eg}[theorem]{Example}
\newtheorem{remark}[theorem]{Remark}
\newtheorem{conjecture}{Conjecture}[section]
\newtagform{clai}{section}{Claim}
\newtheorem*{claim}{\textbf{Claim}}
\newtheorem{obs}[theorem]{Observation}

\numberwithin{equation}{section}

\numberwithin{mytheorem}{subsection}
\newcommand{\norm}[1]{\vert \vert {#1} \vert \vert}
\def \no {\noindent}

\newcommand \R{\mathbb{R}}

\newcommand \D{\mathcal{D}}
\newcommand \supo{\operatorname{}}
\newcommand \unit{1_{\ring{\C}}}
\newcommand \C {\mathcal{C}}
\newcommand \U {\mathcal{U}}
\newcommand \rar{\rightarrow}

\newcommand{\nrh}[1]{\operatorname{NRE}\{\ring{#1}\}}
\newcommand{\eq}[1]{\left[{#1}\right]_R^{}}
\newcommand{\bpm}[1]{\operatorname{BPM}\{\ring{#1}\}}
\newcommand{\um}[1]{\operatorname{UM}\{\ring{#1}\}}
\newcommand{\rh}[1]{\operatorname{RH}\{\ring{#1}\}}

\newcommand{\atom}[1]{\mathcal{A}^{\mathcal{U}}_{#1}}
\newcommand{\ring}[1]{\mathcal{R}_{#1}}
\newcommand{\mdsum}[1]{\overline{#1} }
\newcounter{casenum}
\def \GCD{\operatorname{GCD}}
\newenvironment{caseof}{\setcounter{casenum}{1}}{\vskip.5\baselineskip}
\newcommand{\casea}[2]{\vskip.5\baselineskip\par\noindent {\bfseries Case \arabic{casenum}:} #1\\#2\addtocounter{casenum}{1}}
\newcounter{subcasenum}

\usepackage[a4paper,
left=1in,right=1in,top=0.85in,bottom=0.8in,%
footskip=.25in]{geometry}
\usepackage{blindtext}

\title{Neural codes  and  Neural ring endomorphisms}
\author{Neha Gupta}
\author{Suhith K N}
\address{Department of Mathematics\\ Shiv Nadar Institution of Eminence\\ Delhi-NCR}
\email{sk806@snu.edu.in, neha.gupta@snu.edu.in}

\date{September 14, 2023\\ Suhith's research is partially supported by Inspire fellowship under the Inspire program of the Department of Science and Technology(DST),  DST grant IF190980}
\begin{document}
	\maketitle
	\thispagestyle{empty}
\begin{abstract}
	{We investigate combinatorial, topological, and algebraic properties of certain classes of neural codes. We look into a conjecture that states if the minimal \textit{open convex} embedding dimension of a neural code is two, then its minimal \textit{convex} embedding dimension is also two. We prove the conjecture for two interesting classes of examples and provide a counterexample for the converse of the conjecture. We introduce a new class of neural codes, \textit{doublet maximal}. We show that a doublet maximal code is open convex if and only if it is max intersection-complete. We prove that surjective neural ring homomorphisms preserve max intersection-complete property. We introduce another class of neural codes, \textit{circulant codes}. We give the count of neural ring endomorphisms for several sub-classes of this class.	
	 }
\end{abstract}
\no\textbf{Mathematics Subject classification:} 52A37 $ \cdot $ 92B20 $ \cdot $ 54H99 $ \cdot $ 16W20 \\
\noindent{}{\bf KEYWORDS}:\hspace*{0.5em} \emph{}

	\bibliographystyle{ieeetr}
	\pagestyle{plain}
	
	\thispagestyle{empty}
	
	\bibliographystyle{plain}
	\pagestyle{plain}
	
	\thispagestyle{empty}
	
\input{int_paper}
	\input{core}
	\input{nre}

	\bibliography{refs}
	
\end{document}

%% file: int_paper.tex
\section{Introduction}

\quad The Nobel prize in Medicine for the year 2014 was awarded to Neuroscientist John O'Keefe \cite{o1976place}  for the discovery of place cells (type of neurons) in the rat's hippocampus. Place cells respond when an animal is in a particular region in its environment (stimuli space). Different place cells respond in different regions. The regions in which a place cell responds are  called its place field. Place cells and their place fields encode binary information about the responses of an animal in a given environment. So, the study of binary codes is an essential part of this area of research.  

We define a neural code $ \C $ to be a collection of subsets of the set $ [n]=\{1,2,\dots,n\} $. Each element of $ \C $ is called a codeword. Given a collection of place fields one can associate it with a neural code. Consider a collection of place fields, $\U= \{U_1,\dots,U_n\}  $ in some stimuli space  $ X\subseteq \R^k $. Then the associated neural code for $ \U $ is defined as $$ \C(\U) = \left\{\sigma\in [n] \ \Big\vert\ \displaystyle\bigcap_{j\in \sigma} U_j \backslash \displaystyle\bigcup_{i\not\in\sigma } U_i \not= \phi \right\}. $$ We call $ \atom{\sigma}=\displaystyle\bigcap_{j\in \sigma} U_j \backslash \displaystyle\bigcup_{i\not\in \sigma } U_i $ the atom of a codeword $ \sigma\in \C(\U) $,  and denote $ U_\sigma =\displaystyle\bigcap_{j\in \sigma} U_j  $. Fix $ U_\emptyset= X. $  Figure \ref{figrel} discusses an example to obtain a neural code from a given collection of place fields. Conversely, given any neural code $ \C $, one can associate a collection of regions (subsets of some $ \R^k $) that can represent the neural code geometrically. We say that a neural code   $ \C $ is realizable if there exists a collection $ \U=\{U_1^{},\dots,U_n^{}\}$ with $ U_i\subseteq X \subseteq \R^k, $ such that $ \C=\C(\U). $ Here, $ \U $ is called the realization of $ \C$ and $ X $ its stimuli space. Further, we will address neural codes as simply codes for the rest of the paper. Also, we will fix the notation $ X$ for the stimuli space of a realizable code.
   	\begin{figure}[]
   	\begin{tikzpicture}[scale=0.7]
   		\draw (-1.3,-1.4) rectangle (5.4,4.4);
   		\draw (4.6,3.2)node[above] {$X$}; 

   		\draw (1,1) circle (2cm);
   		\draw (0,-0.5)node[above] {$U_1$};
   		\draw(2,2) circle (2cm);
   		\draw (1.9,3)node[above] {$U_2$};
   		\draw(4.3,-0.2) circle (1cm);
   		\draw (4,-0.5)node[above] {$U_3$};
   		\draw(1.9,1) circle(0.7cm);
   		\draw (1.9,0.7)node[above] {$U_4$};
   	\end{tikzpicture}
 
   	\caption{
   		This figure has four place fields in the stimuli space $ X \subseteq \R^2 $. The associated neural code is  $  \{\emptyset,1,2,3,12,124\} $ Technically the code obtained should have been written as $ \{\emptyset,\{1\},\{2\},\{3\},\{1,2\},\{1,2,4\}\} $. But we abuse the notations for simplicity throughout the paper. }
   	\label{figrel}
   \end{figure}
   
    Experimental data \cite{curto2017makes} showed that the place fields are approximately open convex sets in $\R^2$. So, for every realizable code $ \C $, there is a natural question to ask about the topological properties of the subsets $ U_i $ of $ \U. $ The realization $ \U $ of a code $ \C $ is called open convex if each set $ U_i $ is open convex in $ \R^k $. In this case, $ \C $ is referred as an open convex code. Similarly, we can have a convex or a closed convex code. Next, we discuss about minimal embedding dimension of a code.  Let $ \C $ be a realizable code with $ X\subseteq \R^k $ as its stimuli space. Then $ k $ is said to be minimal embedding dimension of the code $ \C, $ if there exists no $ l<k $ with a  collection $ \U' $ in $ \R^l $  such that $ \C(\U')=\C. $  Franke and Muthiah \cite{franke2018every} provided an algorithm to prove that every code is convex. Cruz et al. \cite{cruz2019open} showed that if a code $\C$ is max intersection-complete\footnote{ A code $ \C $ is max intersection-complete if $ \C $ contains all intersections of its maximal codewords.} then it  is both open convex and closed convex.
 
In 2013, Curto et al. \cite{curto2013neural} explored this topic and brought some algebraic direction to it. They associated a ring structure to a given code $ \C $ on $n $ neurons, and called it a neural ring $ \ring{\C} $ associated with $ \C. $ They defined $ \ring{\C} $  as $ \mathbb{F}_2[x_1,x_2,\dots,x_n]/I_\C $ where  $ I_\C=\{f\in\mathbb{F}_2[x_1,x_2,\dots,x_n] \mid f(c)=0 \text{ for all } c\in \C\}. $ For any codeword $ c$ the characteristic function $ \rho_c $ \footnote{ The characteristic function is given by $\rho_c(v)= \begin{cases}
		1 & \text{ if } v=c \\ 0 & \text{ otherwise}.
	\end{cases} $}  has  $  \underset{c_i=1}{\Pi}x_i\underset{c_j=0}{\Pi}(1-x_j) $ as its polynomial form. Curto and Youngs \cite{curto2020neural} discuss ring homomorphisms between two neural rings. They proved that there is a 1-1 correspondence between code maps $ q:\C\rar \D $ and the ring homomorphisms $ \phi:\ring{\D}\rar \ring{\C}. $ The map $ q $ associated with the ring homomorphism $ \phi $ is denoted by $ q_\phi,$ and is called the associated code map. They also showed that $ \ring{\C} \cong \ring{\D} $ if and only if $ |\C|=|\D|. $ That means, the neural ring loses information of the codewords present in the code and only considers the cardinality of the code. So, they defined some more relevant conditions on the ring homomorphisms, and called these maps called neural ring homomorphisms. Further characterized neural ring homomorphisms with the associated code maps. Lastly, they connected the idea of codes being open convex with neural ring homomorphisms.

Curto et al. \cite{curto2013neural} also defined a neural ideal as $ J_\C =\langle \{\rho_c\mid c\not \in \C\}\rangle. $  Neural ideal is closely associated to Stanley-Reisner ideal \cite{miller2004combinatorial}. Jeffs, Omar and Youngs \cite{jeffs2018homomorphisms} tried to get all ring  homomorphisms from $ \mathbb{F}_2[y_1,\dots,y_n] \rar \mathbb{F}_2[x_1,\dots,x_m] $ that preserve neural ideals. They showed that only specific ring homomorphisms satisfy the above condition. Brown and Curto \cite{brown2022periodic} defined periodic codes inspired by the sound localization system of barn owls. These codes have special patterns that signify the periodicity of the stimulus. They showed that, except for some special cases, this code need not be convex. They introduced a concept known as convex closures, a way of adding codewords to make a code convex. Further, they construct convex closure of periodic codes. We worked with a subclass of periodic codes and called them circulant codes with support $ p $. We count neural ring endomorphisms for numerous codes in this class.

In this paper, we have worked with combinatorial and algebraic properties of some specific kinds of codes. In the following two sections, we work with the combinatorial and topological properties of codes. However, the last two sections are exclusively for the algebraic properties.  This paper is structured as follows. In section 2,  we work with a conjecture given by Franke and Muthiah \cite{franke2018every}.  We provide a few classes of examples in Proposition \ref{lemcon} and Remark \ref{remarksec3} that satisfy this conjecture. Also, we give a counterexample for the converse of this conjecture. We introduce a new class of codes called doublet maximal codes in section 3. The main result in section 3 is Theorem \ref{tdmc}, which states, ``If a code is doublet maximal, then it is open convex if and only if it is max intersection-complete.''  In section 4, we see the relationship of two codes being max intersection-complete via a code map between them (Theorem \ref{thmic}).  In the last section, we work with circulant codes with support $ p $. Figure \ref{fig4} summarizes the main results of this section. 

%% file: core.tex
\section{Convex codes in dimension 1 and 2}
Franke and Muthiah \cite{franke2018every} worked on convex codes and wanted to give a direct relation between convex and open convex codes. They gave the following conjecture: 
\begin{conjecture}\cite[Conjecture 2]{franke2018every}
	Suppose $ \C $ is open convex and has a minimal open convex embedding dimension of 2. Then the minimal convex embedding dimension of $ \C $ is 2. \label{conMFSM2}
\end{conjecture}

 This conjecture seems to hold. We may not yet have a proof for it, but we have two classes of examples that satisfy the conjecture. Further, we will try to see if the converse of this conjecture holds. 

\begin{proposition}
	Let $  \C $ be a code containing subset $ \{i,j,k,\sigma\} $, where  $i,j,k\in \sigma\subseteq[n] $  and $ i,j,k $ are all distinct elements in $ [n]. $  Then the minimal convex embedding dimension of such a code $ \C $ is greater than 1. \label{lemcon}
\end{proposition}
\begin{proof}
	Let $ \C $ be a given code containing $ \{i,j,k,\sigma\} $.
	We show that this code cannot be convex realizable in $ \R $. If possible, let it have a convex realization $ \U $ in $ \R. $  Let $ l\in \{i,j,k\}, $ then we observe that  $ U_l \cap \atom{l}\not=\emptyset
	$ and $ U_l\cap\atom{\sigma}\not=\emptyset $ as $l \in \sigma.  $ Since atoms are disjoint $ U_l $ must contain at least two points. However as $ U_l $'s are convex sets in $ \R, $ they must be intervals.  
	
	Without loss of generality we may assume that $ U_i $ is open, $ U_j $ is clopen (neither closed nor open) and $ U_k $ is a closed set.   Fix $ U_i=(a_i,b_i) $ for some $ a_i\not= b_i\in \R. $   Since $ ijk \subseteq \sigma \in \C $ we have $ U_i\cap U_j \cap U_k\not=\emptyset. $ This implies that $ U_i\cap U_j \not =\emptyset. $ Therefore we choose $ a_j $ such that $ a_i < a_j < b_i.$ Also, as  $ \atom{j}\not=\emptyset$, we get  $  U_j \cap \bigcap_{l\not=j} (U_l)^c= U_j\cap \left(\bigcup_{l\not=j} U_l\right)^c =U_j\backslash \bigcup_{l\not=j}U_l\not= \emptyset  $. Further, this implies $ U_j\cap U_i^c \not=\emptyset. $ So, we must have $ b_j \in U_i^c. $ We choose $ b_j > b_i $ and construct $ U_j=(a_j,b_j]. $
	The above construction is shown in figure \ref{figconl}.  Now  $ U_k$ must intersect $ U_i\cap U_j $. Therefore we choose $ a_k $ such that $ a_j<a_k<b_i. $ As $ \atom{k}= U_k\backslash \bigcup_{l\not=k}U_l  \not= \emptyset$ and with similar calculations as above we see  $ U_k\cap (U_i\cup U_j)^c\not=\emptyset $. So we must have $ b_k  $ lying in $ (U_i\cup U_j)^c. $ Hence choose $ b_k>b_j, $ and construct $ U_k=[a_k,b_k]. $ But this gives us that $ U_j \subset U_i \cup U_k, $ leaving $ \atom{j}=\emptyset. $ That is a contradiction to the fact that $ j\in\C=\C(\U). $
	
	 Note that we have constructed $ U_j $ and $ U_k $ to the right of $ U_i. $ The proof is similar even if we construct the sets on the left side of $ U_i. $ Therefore the code cannot be convex realized in dimension 1. Hence the minimal convex embedding dimension is greater than 1. 
\end{proof}

\begin{figure}[]
	\begin{center}
		\begin{tikzpicture}[scale=7]
			\draw[<->, ultra thick] (0,0) -- (1.5,0);
			\foreach \x/\xtext in {0.2/$ a_i $,0.4/,0.6/$ a_j $,0.8/,1/$ b_i $,1.2/$ b_j $,1.4/}
			\draw[thick] (\x,0.5pt) -- (\x,-0.5pt) node[below] {\xtext};
			 
			\draw (0.2,1pt) node[above] {$U_i$};
			\draw[{(-)}, ultra thick, green] (0.2,0) -- (1.0,0); \draw (0.6,1pt)node[above] {$U_j$};
			
			\draw[{(-]}, ultra thick, red] (0.6,.0) -- (1.2,0.00);
		\end{tikzpicture}
	\end{center}
	\caption{This figure gives us the construction of $ U_i, U_j $ of the Proposition \ref{lemcon}} \label{figconl}
\end{figure}
\begin{proposition}
		Let $ \C '$ be a code containing the subsets $ \{i,j,k,\sigma_{ij},\sigma_{ik},\sigma_{jk}\}  $ where  $ i,j,k $ are distinct elements of $ [n] $ and $i,j\in \sigma_{ij}\subseteq[n] $ with $ k\notin \sigma_{ij}$, and similarly for $ \sigma_{ik},\sigma_{jk}. $ Then the minimal convex embedding dimension for the code $ \C' $ is greater than 1. \label{lemcon2}
\end{proposition}
\no Proof of Proposition \ref{lemcon2} is similar to proof of Proposition \ref{lemcon}.
\begin{remark}\label{remarksec3}
	 Thus we establish two classes of examples
	\begin{align*}
		\C \supseteq& \{i,j,k,\sigma\} \qquad \left(i,j,k\in \sigma\subseteq[n]\right) \\  \C' \supseteq & \{i,j,k,\sigma_{ij},\sigma_{ik},\sigma_{jk}\} \qquad (\text{as defined above})
	\end{align*} 
that have minimal convex embedding dimension greater than 1. So, if $ \C $ (or $\C' $) has a minimal \textit{open} convex embedding dimension 2, then $ \C $ (or $\C' $) is a supporting class of example for the Conjecture \ref{conMFSM2}. 
\end{remark}
\begin{remark}
	Jeffs \cite{jeffs2019sunflowers} defined sunflower to be a collection of sets $ \{U_1,U_2,\dots,U_n\} $ such that $ U_i\cap U_j $ is nonempty and a constant subset for all $ i\not=j. $ The code we obtain from a convex open sunflower with $ n\geq 3 $ always contains $ \{1,2,3,123\dots n\}. $ Hence the codes obtained from open convex sunflowers always satisfy the hypothesis of Proposition \ref{lemcon}. However, the converse may not be true. For example, the realization of the code $ \{1,2,3,4,123\} $  is not a sunflower but the code satisfies hypothesis of Proposition \ref{lemcon}.
\end{remark}
\begin{eg}
	The code $ \C =\{1,2,3,123\}$ has a convex realization in dimension $ 2 $  (Fig. \ref{figa}). The stimuli space of this code is $ X=U_1\cup U_2\cup U_3 $. By Proposition \ref{lemcon}, this code has no convex realization in dimension 1. Thus the code $ \C $ cannot have an open convex realization in dimension 1. Hence the minimal convex embedding dimension must be 2 for the code $ \C. $ Moreover, the sets $ U_i $ are open (Fig. \ref{figa}) . Thus the same figure gives us an open convex realization for the code $ \C $ in dimension 2.  Therefore we have the minimal open convex embedding dimension as 2 for this code.

\end{eg}
Let us now look at the converse statement of Conjecture \ref{conMFSM2} which states that if the minimal convex embedding dimension of a code is 2, then its minimal open convex embedding is also 2.  Consider the code, $ \C=\{\emptyset,1,2,3,4,6,45,56,123\}. $ The figure shown below (Fig. \ref{figb}) gives a convex realization of the code in $ \R^2 $. Note that in the figure the right-most boundary of $ U_4 $ is included.  Moreover, all the other sets are open in $ \R^2. $ As $ \{1,2,3,123\}\subseteq \C $ by Proposition \ref{lemcon}, $ \C $ does not have a convex realization in $ \R. $ Hence the minimal \textit{convex} embedding dimension for $ \C $ is 2. Further, $ \D=\{\emptyset,4,6,45,56\}\subseteq \C.$ Jeffs \cite[Example 2.1]{jeffs2020morphisms} showed that $ \D $ cannot be an open convex code. Therefore  $ \C $ cannot have an open convex realization. Hence $ \C $ serves as a counterexample for the converse of the Conjecture \ref{conMFSM2}.

\begin{figure}[]
	\begin{subfigure}[b]{0.4\linewidth} \centering
		\begin{tikzpicture}[scale=1]	
			\draw[dashed] (1,1) rectangle (2,3) ;\draw[dashed](0,2) rectangle (2,3);\draw[dashed](1,2) rectangle (3,3);
			\draw [<->] (0,3.2)--(2,3.2);
			\draw [<->] (0.8,3)--(0.8,1);
			\draw [<->] (1,1.8)--(3,1.8);
			\draw (1,3.2) node[above] {$ U_1 $};
			\draw (0.8,1.5) node[left] {$ U_2 $};
			\draw (2.5,1.8) node[below] {$ U_3 $};
		\end{tikzpicture}
		\caption{ $ \C=\{1,2,3,123\} $}
		\label{figa}
	\end{subfigure}
	\begin{subfigure}[t]{0.5\linewidth} \centering
		\begin{tikzpicture}[scale=1]
			
			\draw[dashed] (1,1) rectangle (2,3) ;\draw[dashed](0,2) rectangle (2,3);\draw[dashed](1,2) rectangle (3,3);
			\draw[dashed] (4,2) rectangle (8,3);
			\draw[dashed] (7,2) -- (7,3);
			\draw[dashed] (5,2) -- (5,3);
			\draw (6,2) -- (6,3);
			\draw [<->] (6,3.2)--(8,3.2);
			\draw [<->] (4,3.2)--(6,3.2);
			\draw [<->] (5,1.8)--(7,1.8);
			\draw [<->] (0,3.2)--(2,3.2);
			\draw [<->] (0.8,3)--(0.8,1);
			\draw [<->] (1,1.8)--(3,1.8);
			\draw (1,3.2) node[above] {$ U_1 $};
			\draw (0.8,1.5) node[left] {$ U_2 $};
			\draw (2.5,1.8) node[below] {$ U_3 $};
			\draw (6,1.8) node[below] {$ U_5 $};
			\draw (5,3.2) node[above] {$ U_4 $};
			\draw (7,3.2) node[above] {$ U_6 $};
			\draw (-0.5,0.5) rectangle (8.5,4);
			\draw (0,3.5) node[above] {$ X $};
		\end{tikzpicture}
		\caption{ $ \C=\{\emptyset,1,2,3,4,6,45,56,123\}. $ }
		\label{figb}
	\end{subfigure}
	\caption{}
\end{figure}

\section{Doublet maximal codes}

A codeword $ \sigma  $ is said to be 
\textit{maximal} if it is not contained in any other codeword of $ \C. $ In other words, if there exists $ \tau \in\C $ such that $ \sigma\subseteq \tau, $ then $ \sigma=\tau. $  Maximal codewords play an important role. We will see that atoms corresponding to maximal codewords have special properties. The following lemma gives us one such. 

\begin{lemma}
	Let $ \tau\in \C $ be a maximal codeword, and let $ \C $ have a convex realization, $ \U=\{U_1,U_2,\dots,U_n\} $  in $ \R^m, $  then  \label{mopen}
	\begin{enumerate}
		\item $ U_\tau \subseteq \left(\displaystyle\bigcup_{i\not\in \tau} U_i \right)^c  $ \label{mopen1}
		\item If all $ U_i $'s are open in $ \R^m $ (i.e., $ \C $ is open convex) then $ \atom{\tau} $ is open in $ \R^m $.
	\end{enumerate}
\end{lemma}
\begin{proof} 
	Let $ \C $ be the given code with the convex realization $ \U=\{U_1,U_2,\dots,U_n\} $ and let $ X $ be its stimuli space.
	\begin{enumerate} 
		\item If $ \tau=[n], $ then $ \left(\bigcup_{i\notin \tau}U_i\right)^c=U_\emptyset=X. $ This implies the result is true, trivially. Let $ \tau $  be some other maximal codeword in $ \C $ such that $ U_\tau \not\subseteq \left(\displaystyle\bigcup_{i\not\in \tau} U_i \right)^c. $ Then there exists some  $ x\in U_\tau $ and $ x \not \in \left(\displaystyle\bigcup_{i\not\in \tau} U_i \right)^c. $ This implies that $ x \in \displaystyle\bigcup_{i\not\in \tau} U_i.$ Therefore there exists   $ k\not \in \tau $ such that $ x \in U_k. $ Define a codeword $ \beta $  such that $  \supo{\beta}= \{i\in [n] \mid i\not\in \tau \text{ and } x \in U_i\}. $ Thus clearly $ \beta \not =\emptyset. $ Denote $ \alpha= \tau\cup \beta  $. Since $ x\in U_{i} $ for all $ i \in \beta $ we have that $ x\in U_\beta.$ This implies  $x\in U_\tau\cap U_\beta = U_\alpha. $ Also, $ x\not\in \bigcup_{i\not \in \alpha} U_i,  $ as  $ \alpha $ contains exactly those $ i $'s for which $ x\in U_i. $ Therefore $ x\in U_\alpha\backslash \displaystyle \bigcup_{i\not \in \alpha}U_i= \atom{\alpha}. $ Hence as $ \atom{\alpha}\not=\emptyset $, we have $ \tau \subsetneq \alpha \in \C(\U)=\C, $ which contradicts the maximality of $ \tau. $ Hence the proof.
		\item We know that $ \atom{\tau}= U_\tau \Big\backslash \displaystyle \bigcup_{i\not \in \tau}U_i = U_\tau \cap \left(\displaystyle\bigcup_{i\not\in \supo{\tau}} U_i \right)^c.   $ Using part (\ref{mopen1}) we have $ \atom{\tau}= U_\tau. $ Since finite intersection of open (or closed) sets is open (or closed) we have the proof.
	\end{enumerate}
\end{proof} 
\begin{remark}
In the case when the code is closed convex the atom of maximal codeword will be a closed set. The proof of this is similar to part 2 of Theorem \ref{mopen1}.
\end{remark}
\noindent Next, we work with codes called max intersection-complete. Cruz et al. \cite{cruz2019open} defined max intersection-complete codes as follows. 
\begin{definition}
	The intersection completion of a code $ \C $ is the collection of all non-empty intersections of codewords in $ \C:  $ $$\widehat{\C}= \left\{\sigma \ \Big\vert\ \sigma  = \bigcap_{v\in \C'} v \text{ for some non-empty sub-code } \C' \subseteq \C \right\}.$$ Denote $ M(\C)  $ to be the collection of all maximal codewords of $ \C. $ Note that $\bigcup_{\sigma\in M(\C)}\sigma=[n]$. A code $ \C $ is said to be \textit{max intersection-complete }if $ \widehat{M(\C)} \subseteq \C. $ For example, if  $ M(\C)=\{\tau_1^{},\tau_2^{}\},$ then $ \C  $ will be max intersection-complete, if and only if $ \tau_1^{}\cap \tau_2^{} \in \C. $ 
\end{definition}
Cruz et al. \cite{cruz2019open} showed that the codes that are max intersection-complete are both open convex and closed convex. Also, they gave an upper bound for the minimal embedding dimension of such codes. We look at the converse of the theorem,  i.e., whether open convex codes are max intersection-complete? The code $ \C=\{3,5,12,13,14,45,123,124,145\}, $ is open convex in dimension 1, but it is not max intersection-complete.  Figure \ref{figmip} has further details of this code $ \C $. We observed that having 3 maximal codewords did break the converse. Hence we propose the following result. 
\begin{figure}[]
	\centering
	\begin{tikzpicture}[scale=5]
		\draw[<->, ultra thick] (-0.3,0) -- (2.1,0);
		\foreach \x/\xtext in {-0.2/0,0/1,0.2/2,0.4/3,0.6/4,0.8/5,1/6,1.2/7,1.4/8,1.6/9,1.8/10,2.0/11}
		\draw[thick] (\x,0.5pt) -- (\x,-0.5pt) node[below] {\xtext};
		\draw (0.3,1pt) node[above] {$U_1$};
		\draw[{(-)}, ultra thick, green] (0.3,0) -- (1.0,0); \draw (0.4,1pt)node[above] {$U_2$};
		\draw[(-), ultra thick, red] (0.4,0) -- (0.7,0);
		\draw (0.15,1pt)node[above] {$U_3$};
		\draw[(-), ultra thick, blue] (0.15,0) -- (0.5,0);
		\draw (0.6,1pt)node[above] {$U_4$};
		\draw[(-), ultra thick, yellow] (0.6,.0) -- (1.6,0);
		\draw (0.8,1pt)node[above] {$U_5$};
		\draw[(-), ultra thick, brown] (0.8,.0) -- (1.8,0);
		\draw (0.9,0.2)node[above] {$ \atom{145} $};
		\draw[->,  thick, black] (0.9,.0) -- (0.9,0.2);
		\draw (0.75,-0.2)node[below] {$ \atom{14} $};
		\draw[->,  thick, black] (0.75,.0) -- (0.75,-0.2);
		\draw (0.65,0.2)node[above] {$ \atom{124} $};
		\draw[->,  thick, black] (0.65,.0) -- (0.65,0.2);
		\draw (0.55,-0.2)node[below] {$ \atom{12} $};
		\draw[->,  thick, black] (0.55,.0) -- (0.55,-0.2);
		\draw (0.45,0.2)node[above] {$ \atom{123} $};
		\draw[->,  thick, black] (0.45,.0) -- (0.45,0.2);
		\draw (0.35,-0.2)node[below] {$ \atom{13} $};
		\draw[->,  thick, black] (0.35,.0) -- (0.35,-0.2);
		\draw (0.25,0.2)node[above] {$ \atom{3} $};
		\draw[->,  thick, black] (0.25,.0) -- (0.25,0.2);	 
		\draw (1.3,-0.2)node[below] {$ \atom{45} $};
		\draw[->,  thick, black] (1.3,.0) -- (1.3,-0.2);
		\draw (1.7,-0.2)node[below] {$ \atom{5} $};
		\draw[->,  thick, black] (1.7,.0) -- (1.7,-0.2);
	\end{tikzpicture}
	\caption{This figure gives a code $ \C=\C(\U)= \{3,5,12,13,14,45,123,124,145\}$ realized by $ \{U_1,U_2,U_3,U_4,U_5\}. $ The code $ \C $ is open convex in dimension 1 and $ {123,145} $ are maximal codewords, whose intersection is $ 1 $ and 1 doesn't belong to $ \C $.  }
	\label{figmip}
\end{figure}

\begin{theorem} \label{tcmip}
	Let $ \C $ be a code that contains the empty set as a codeword along with exactly two maximal codewords. Then $ \C $ is open convex if and only if $ \C $ is max intersection-complete.
\end{theorem}

\begin{proof}

	Let $ M\{\C\}=\{\tau_1,\tau_2\} $. We know by Theorem 1.2 of \cite{cruz2019open} that if $\C$ is max intersection-complete then $\C$ is both open convex and closed convex. So we already have the proof for the necessary condition. Next, the proof for sufficient condition consider $ \C $ to be an open convex code. We will show that $ \C $ is max intersection-complete. Let $ \sigma=\tau_{1}\cap \tau_{2}. $ If $ \sigma= \tau_1^{} \cap \tau_2^{}=  \emptyset  $ then $ \widehat{M(\C)}=\{\emptyset\} \subseteq \C $.   Hence in this case $ \C $ is max intersection-complete. Next, assume $ \sigma\not=\emptyset. $  Let $ \U=\{U_1,\dots,U_n\} $ be a collection of open convex sets in some $ \R^m $ such that, $ \C(\U)=\C. $ For $ \C $ to be max intersection-complete, we need to show $ \sigma \in \C. $ Suppose not. Then, as $ \sigma \not \in \C= \C(\U), $ the atom of $ \sigma $, $ \atom{\sigma}=\emptyset. $ Therefore,
	\begin{equation}\label{key}
		\displaystyle U_\sigma \Big \backslash \displaystyle\bigcup_{j\not\in \supo{\sigma} } U_j= \emptyset \implies U_\sigma \subseteq \displaystyle\bigcup_{j\not\in \supo{\sigma} } U_j
	\end{equation}
	Also, $ \atom{\tau_i^{}}= U_{\tau_i^{}} $ for $ i=1,2$ by Lemma \ref{mopen}. Next, we show that $ U_{\tau_1^{}}  $ and $ U_{\tau_2^{}}  $ form a separation\footnote{A separation of $ X $ is a pair $ U,V $ of disjoint nonempty open subsets of $ X $ whose union is $ X $. The space $ X $ is not connected if there exist a separation. } of $ U_\sigma. $ 
	
	As $ \tau_1^{},\tau_2^{} \in \C=\C(\U), $ we have $ \atom{\tau_{1}} \not = \emptyset $ and $\atom{\tau_2^{}}\not = \emptyset. $ Consequently, $ U_{\tau_{1}} $ and $ U_{\tau_{2}} $ are non-empty. Further, as $ \atom{\tau_{1}} \cap\atom{\tau_{2}}=\emptyset$ we have $ U_{\tau_1^{}} \cap U_{\tau_2^{}}=\emptyset. $ Moreover, $ U_\sigma, U_{\tau_{1}^{}} $ and $ U_{\tau_{2}^{}} $ are open in $ \R^m $ as they are the finite intersection of open sets.  Also, for $ i=1,2$ we have $ U_{\tau_i^{}}= U_{\tau_i^{}} \cap U_\sigma $ as $ U_{\tau_i^{}}\subseteq U_\sigma $. So, $ U_{\tau_{1}^{}} $ and $ U_{\tau_{2}^{}} $ are open in $ U_\sigma. $ Therefore it is only left for us to prove that $ U_\sigma= U_{\tau_1^{}} \cup U_{\tau_2^{}}. $ Observe that $$ U_{\tau_1^{}} =\displaystyle\bigcap_{j\in \tau_1^{}} U_j= \displaystyle\bigcap_{j\in \sigma} U_j \ \cap\ \displaystyle\bigcap_{\substack{{i\in \tau_1^{} }\\ { i \not \in \sigma } }}U_j= U_{\sigma} \cap U_{\tau_1^{}\backslash \sigma}.  $$ Similarly $ U_{\tau_2^{}}= U_\sigma \cap U_{\tau_2^{}\backslash \sigma}.$ Thus,  $ U_{\tau_1^{}} \cup \ U_{\tau_2^{}} = \left(U_{\sigma} \cap U_{\tau_1^{}\backslash \sigma}\right) \cup \left(U_{\sigma} \cap U_{\tau_2^{}\backslash \sigma}\right)= U_\sigma \cap \left(U_{\tau_1^{}\backslash \sigma} \cup U_{\tau_2^{}\backslash \sigma}\right). $ 
	\begin{claim}
		$ U_\sigma \subseteq \left(U_{\tau_1^{}\backslash \sigma} \cup U_{\tau_2^{}\backslash \sigma}\right). $ Suppose not. Then there exists an $ x \in U_\sigma $ such that  $ x\not\in \left(U_{\tau_1^{}\backslash \sigma} \cup U_{\tau_2^{}\backslash \sigma}\right).$ So, $ x\not \in U_{\tau_1^{}\backslash \sigma} \text{ and } x\not \in U_{\tau_2^{}\backslash \sigma}.    $ But, from Equation \ref{key},  $ x\in \bigcup_{j\notin\sigma} U_j$. Thus $ x\in U_k $ for some $ k\notin\sigma. $ Note that, $$k\notin\sigma \implies k\in [n]\backslash \sigma \implies k\in (\tau_1\cup\tau_2 )\backslash \sigma $$ This implies there exists a $ k \in (\tau_{1}\backslash \sigma)\cup (\tau_{2}\backslash \sigma)$ such that $ x\in U_k. $ But this is a contradiction to the fact that $  x\not \in U_{\tau_1^{}\backslash \sigma} \text{ and } x\not \in U_{\tau_2^{}\backslash \sigma}. $ Hence the supposition is wrong, implying $ U_\sigma \subseteq \left(U_{\tau_1^{}\backslash \sigma} \cup U_{\tau_2^{}\backslash \sigma}\right). $
	\end{claim}
\no	By the claim we get $ U_{\tau_1^{}} \cup\ U_{\tau_2^{}}= U_\sigma. $  This means that $ U_{\tau_1^{}}  $ and $  U_{\tau_2^{}} $ form a separation of $ U_\sigma.  $ But $ U_\sigma $ is intersection of connected sets so it must be a connected set itself. Hence cannot have a separation. Thus  $ \sigma\in \C(\U)= \C. $
\end{proof}
\begin{remark}
 The above theorem holds for closed convex. Also, in the proof the only difference is that the separation comes from closed sets. 
\end{remark}
\begin{eg} \label{eg}
	Consider the sets $ \U= \{U_1,U_2,U_3,U_4,U_5,U_6\} $ in $ \R $ as in Figure \ref{ex2}.  Let $ \C= \C(\U)=\{\emptyset,2,4,12,23,45,46\}. $  The code $ \C $ has 4 maximal codewords. Moreover, $ \C $ is  max intersection-complete as well as open convex. But the interesting fact is that one can split the nonempty codewords into $ \C=\C_1 \cup\ \C_2, $ where $ \C_1= \{\emptyset,2,12,23\} $ and $ \C_2=\{\emptyset,4,45,46\}. $ The codes $ \C_1 $ and $ \C_2 $ satisfy the hypothesis of the Theorem \ref{tcmip}. This leads us to define a new class of codes called doublet maximal codes.
	\begin{figure}[]
		\centering
		\begin{tikzpicture}[scale=5]
			\draw[<->, ultra thick] (-0.3,0) -- (2.1,0);
			\foreach \x/\xtext in {-0.2/0,0/1,0.2/2,0.4/3,0.6/4,0.8/5,1/6,1.2/7,1.4/8,1.6/9,1.8/10,2.0/11}
			\draw[thick] (\x,0.5pt) -- (\x,-0.5pt) node[below] {\xtext};
			\draw (0.5,-3.7pt) node[below] {$U_2$};
			\draw[{(-)}, ultra thick, green] (0.3,-0.13) -- (0.8,-0.13); \draw (0.42,0.5pt)node[above] {$U_1$};
			\draw[(-), ultra thick, red] (0.3,0) -- (0.5,0);
			\draw (0.8,1pt)node[above] {$U_3$};
			\draw[(-), ultra thick, blue] (0.5,0) -- (0.8,0);
			\draw (0.96,-6pt)node[above] {$U_4$};
			\draw[(-), ultra thick, yellow] (1,-0.13) -- (1.8,-0.13);
			\draw (1,1pt)node[above] {$U_5$};
			\draw[(-), ultra thick, brown] (1,.0) -- (1.3,0);
			\draw (1.8,1pt)node[above] {$U_6$};
			\draw[(-), ultra thick, blue] (1.3,.0) -- (1.8,0);
			\draw (1.1,0.2)node[above] {$ \atom{45} $};
			\draw[->,  thick, black] (1.1,.0) -- (1.1,0.2);
			
			\draw (0.65,0.2)node[above] {$ \atom{23} $};
			\draw[->,  thick, black] (0.65,.0) -- (0.65,0.2);
			\draw (0.5,0.2)node[above] {$ \atom{2} $};
			\draw[->,  thick, black] (0.5,.0) -- (0.5,0.2);

			\draw (0.35,0.2)node[above] {$ \atom{12} $};
			\draw[->,  thick, black] (0.35,.0) -- (0.35,0.2);

			\draw (1.3,0.2)node[above] {$ \atom{4} $};
			\draw[->,  thick, black] (1.3,.0) -- (1.3,0.2);
			\draw (1.62,0.2)node[above] {$ \atom{46} $};
			\draw[->,  thick, black] (1.62,.0) -- (1.62,0.2);
			\draw (2,-0.2)node[above] {$ X=\R $};
				\draw (1.95,0.2)node[above] {$ \atom{\emptyset} $};
					\draw[->,  thick, black] (1.95,.0) -- (1.95,0.2);
		\end{tikzpicture}
		\caption{This figure gives a code $ \C= \{\emptyset,2,4,12,23,45,46\}.$  }
		\label{ex2}
	\end{figure}
\end{eg}
\begin{definition}[doublet maximal codes]
	A code $ \C $ is  called a \textit{doublet maximal} if $M(\C)=\{\tau_i^{}\}_{i \in [p]}, $ the collection of all maximal codewords of $ \C $, have the property that for every $ i\in [p] $ there exists at most one $ j\not= i $ such that $ \tau_i\cap \tau_j\not=\emptyset. $ 
\end{definition}
\begin{eg}\label{egw}
	\begin{enumerate}
		\item Let $ \C_1 =\{\emptyset,2,4,12,23,45,46\}.$ This is a doublet maximal code with two pairs of maximal codewords $ \{12,23\} $ and $ \{45,46\}. $
		\item Let $ \C_2=\{\emptyset,2,4,12,23\}.$ This is a doublet maximal code with one pair, $ \{12,23\} $ and and one singleton, $ \{4\} $ as maximal codewords. 
		\item Let $ \C_3 =\{3,5,12,13,14,45,123,124,145\}. $ This is a non-example. This code has 3 maximal codewords with all pairwise intersections being non-empty. Also, from Figure \ref{figmip} we can see that this code is not max intersection-complete. 
	\end{enumerate}
\end{eg}
\begin{remark} \label{rdmc}
	The code $\C_1$ in Example \ref{egw} is both open convex and max intersection-complete. Naturally, one wants to know if this is true for all doublet maximal codes. We have successfully generalized Theorem \ref{tcmip} to all doublet maximal codes. Before we state the generalization we introduce restriction of a code. 
\end{remark}
\begin{definition}
	Let $\C$ be a code on $n$ neurons and $\Gamma\subseteq \C$. Then the restriction of the code $\C$ to $\Gamma$ is defined as,
	$$\C\vert_\Gamma=\{\alpha\in \C\mid \alpha\subseteq \gamma \text{ for some } \gamma\in\Gamma \}.$$ For example let $\C=\{3,4,12,34,123,345\}$ and $\Gamma=\{34,123\}$ then $\C\vert_\Gamma=\{3,12,34,123\}.$
\end{definition}
\no Now, we give the generalization of Theorem \ref{tcmip} in the following result.
\begin{theorem} \label{tdmc}
	Let $ \C $ be a doublet maximal code with $ \emptyset \in \C $ then $ \C $ is open (or closed) convex if and only if $ \C $ is max intersection-complete.
\end{theorem}

\begin{proof}
	Let $\C$ be a doublet maximal code. Assume $\C$ to be open convex. We will show that $\C$
	is max intersection-complete.  The strategy for proving the sufficient condition is to use Theorem \ref{tcmip}, iteratively. We now discuss the details. Let $\U=\{U_1,\dots,U_n\}$ be an open convex realization of $\C$ with $U_i\subseteq \R^k$.
	  Let $R=\{(\sigma,\tau)\in M(\C)\times M(\C)\mid \sigma\cap\tau \not=\emptyset\}.$ Then $R$ defines an equivalence relation on $ M(\C).$ Let $[\sigma]_R^{}$ denote the equivalence class of $\sigma\in M(\C)$ with respect to $R$. Then, by the definition of doublet maximal codes,  $ \big\vert\eq{\sigma}\big\vert\in\{1,2\}$. Let $\big\vert\{\eq{\sigma}\mid\sigma\in M(\C)\}\big\vert=m$. Let us choose some order on $\{\eq{\sigma}\mid\sigma\in M(\C)\} $ and write it as  $\left\{\eq{\sigma^1},\dots,\eq{\sigma^m}\right\}.$We further partition $\U$ into $\{\U_1,\dots,\U_m\}$, where $$\U_i=\left\{U_j\ \Big\vert\  j\ \in\bigcup_{\alpha\in\eq{\sigma^i}} \alpha\right\}.$$ Note that $\C(\U_i)$ is an open convex sub-code of $\C$ with  $M(\C(\U_i))=\eq{\sigma^i}$. So,  $\C(\U_i)=\C|_{\eq{\sigma^{i}}}$. Observe  $\emptyset\in \C(\U_i)$ for all $i\in[m]$. Moreover, $\C=\cup_{i=1}^m \C(\U_i)$ since $\{\U_1,\dots,\U_m\}$ is a partition of $\U$. For all $i\in[m], \vert M(\C(\U_i))\vert\in\{1,2\}$ and $\emptyset\in\C(\U_i)$ implies that each $\C(\U_i)$ satisfy the hypothesis of Theorem \ref{tcmip}. Moreover,  $\C(\U_i)$'s are all open convex. Hence by Theorem \ref{tcmip}, for all $i\in[m],$ $\C(\U_i)$ is max intersection-complete.  
	
	Finally, we will show that $\C$ is max intersection-complete. For that, we will show that $\widehat{M(\C)}\subseteq \C$. Let $M'\subseteq M(\C)$. We will consider various cases for $M'$ and show that in each case, $\cap_{v\in M'}\ v \in \C. $ Note that if $ |M'|\leq1$, there is nothing to prove. Further cases are as follows:
	
	\no\textbf{Case 1:} $|M'|> 2$. Then by the definition of doublet maximal code $\cap_{v\in M'}\ v=\emptyset$. Since $\emptyset\in\C$, we are done for this case. \\
	  \no \textbf{Case 2: }  $|M'|=2$. Let $M'=\{\tau_1,\tau_2\}$.Then $ \cap_{v\in M'}\ v =\tau_1\cap\tau_2.$ We have following two sub-cases: \\
	  \textbf{Case 2a:} There exists an $i\in[m]$ such that $\tau_1,\tau_2\in\eq{\sigma^i}.$ In this case $\tau_1$ and $\tau_2$ are the only maximal codewords of the code $\C(\U_i)$. Since $\C(\U_i)$ is max intersection-complete, $\tau_1\cap\tau_2\in\C(\U_i)\subseteq\C$. Hence the case.\\
	  \textbf{Case 2b:} There exist $i,j\in[m]$ with $i\not=j$  such that $\tau_1\in\eq{\sigma^i}$ and $\tau_2\in\eq{\sigma^j}$. Since $i\not=j$ we have $\tau_1\cap\tau_2=\emptyset.$ Thus, in this case, $\tau_1\cap\tau_2\in\C$. Hence the case.\\
	  Therefore, for given any $M'\subseteq M(\C),$ $\cap_{v\in M'}\ v\in\C$. Thus $\widehat{M(\C)}\subseteq\C$. Hence $\C$ is max intersection-complete.
	\\ The proof for the necessary condition comes directly from Theorem 1.2 of \cite{cruz2019open} which states that if $\C$ is max intersection-complete then $\C$ is both open convex and closed convex.
\end{proof}
So far, we have studied the type of codewords in a code and captured that essence to connect it
with the topological properties of the code, like open convex and closed convex. However, in the
remaining part of the paper, we will work in the algebraic direction of codes. We will explore the algebraic tools developed over the past decade to study codes, like neural rings and neural ring homomorphisms. Then connect the code’s algebraic properties to the code’s properties based on the type of codewords, like max intersection-complete. Moreover, in the remaining part of our paper, we will work with the binary form of the codewords instead of the set form used in previous sections. The binary form makes it easier for us to work with neural rings. Also, we will use the same binary form in section 5 where we define specific matrices based on this form. This will help us observe some exciting results.

\section{Neural ring homomorphisms and max intersection-complete codes}
\subsection{Background and Preliminaries}
In this section, we consider the  codewords in their binary form. For any $ c\in \C $ we will write $ c=c_1^{}c_2^{}\cdots c_n^{}, $ where $ c_i^{} $ is $ 1 $ if $ i\in c $ and 0 otherwise. This is same as seeing $ \C\subset\{0,1\}^n. $ Curto and Youngs \cite{curto2020neural}	gave description of neural ring homomorphisms as follows
\begin{definition}
	Let $ \C \subset \{0,1\}^n $ and $ \D\subset \{0,1\}^m $ be codes, and let $ \ring{\C}= \mathbb{F}_2[y_1,\dots,y_n]/I_{\C} $ and $ \ring{\D}= \mathbb{F}_2[x_1,\dots,x_m]/I_{\D}  $ be the corresponding neural rings.  A ring homomorphism $ \phi:\ring{\D}\rar \ring{\C} $ is a neural ring homomorphism if $ \phi(x_j)\in\{y_i\mid i \in [n]\} \cup \{0,1\} $ for all $ j\in [m],$ where $ x_i=\displaystyle\sum_{\{d\in\D\mid d_i=1\}} \rho_d $. A neural ring homomorphism $ \phi $ is a neural ring isomorphism if it is a ring isomorphism and its inverse is also a neural ring homomorphism. 
\end{definition}    
  At the beginning of their paper, Curto and Youngs \cite{curto2020neural}   discuss ring homomorphisms between two neural rings. They proved that there is a 1-1 correspondence between code maps $ q:\C\rar \D $ and the ring homomorphisms $ \phi:\ring{\D}\rar \ring{\C}. $ The map $ q, $ associated with the ring homomorphism $ \phi $ is denoted by $ q_\phi. $   Later, the authors classify all the neural ring homomorphisms using the following theorem: 
\begin{theorem}\cite[Theorem 3.4]{curto2020neural} \label{thmnycc}
	A map $ \phi:\ring{\D}\rar \ring{\C} $ is a neural ring homomorphism if and only if $ q_\phi $ is a composition of the following elementary code maps:
	\begin{enumerate}
		\item Permutation 
		\item Adding a trivial neuron (or deleting a trivial neuron)
		\item Duplication of a neuron (or deleting a neuron that is a duplicate of another)
		\item Neuron projection (or deleting a not necessarily trivial neuron)
		\item Inclusion (of one code into another)
	\end{enumerate}
	Moreover, $ \phi $ is a neural ring isomorphism if and only if $ q_\phi $ is a composition of maps $ (1)-(3). $   
\end{theorem}
Lastly, Curto and Youngs \cite{curto2020neural} connected the idea of codes being open convex with neural ring homomorphisms using the following theorem,
\begin{theorem}\cite[Theorem 4.3]{curto2020neural} \label{thnrh}
	Let $ \C $ be a code containing the all-zeros codeword and $ q:\C \rar\D $ a surjective code map corresponding to a neural ring homomorphism. Then if $ \C $ is convex (open convex), $ \D $ is also convex (open convex) with $ d(\D) \leq d(\C)$\footnote{$ d(\C) $  is used by the authors to denote the minimal open convex embedding dimension of the code $ \C. $}.  
\end{theorem}

\begin{remark}
We observe that the above theorem holds for closed convex codes too. The proof can be obtained similar to the original version given by Curto and Youngs \cite{curto2020neural}.  
\end{remark}
\subsection{Main Theorem}
Now we will try to connect neural ring homomorphisms with the max intersection-complete property. For the remainder of the section, we assume that $ \C $ is a code on $ n $ neurons and the number of neurons of code $ \D $ will be specified if and when required.  
\begin{obs} \label{obssig}
	Let $ q:\C \rar \D $ be a code map corresponding to a given neural ring homomorphism $ \phi:\ring{\D}\to \ring{\C}. $ If $ \sigma \subseteq \tau $ in $ \C $ then $ q(\sigma)\subseteq q(\tau)  $ in $ \D. $
	
 This observation is fairly computational and can be obtained by applying any of the five maps of Theorem \ref{thmnycc} to an arbitrary codeword of $\C$.
\end{obs}
\begin{lemma}
	Let $ q:\C \rar \D $ be either a permutation, or adding/ deleting a trivial or duplicate neuron, then  $\tau \in \C $ is a maximal codeword if and only if $ q(\tau) \in \D $ is a maximal codeword. \label{maxiso}
\end{lemma}
\begin{proof}
	If $ q $ is either a permutation, or adding/ deleting a trivial or duplicate neuron then the corresponding neural ring homomorphism is an isomorphism. This implies that $ q $ is a bijection \cite[Proposition 2.3]{curto2020neural}. 
	
	Let $ \tau\in \C $ be a maximal codeword. Suppose  $ q(\tau) $ is not a maximal codeword in $ \D $. Then there exists $ q(\lambda)\in \D $ such that $ q(\tau)\subsetneq q(\lambda). $ This implies $ \tau \subsetneq \lambda $ as $ q $ is a bijection. This is a contradiction to the fact that $ \tau  $ is a maximal codeword in $ \C. $
	
	Conversely, if $ q(\tau) $ is maximal codeword in $ \D. $ Then one can show that $ \tau $ is a maximal codeword in $ \C $ using $ q^{-1} $ and the same idea as used in the first part of the proof. This works because $ q^{-1} $ is again either a permutation, or adding / deleting a trivial or duplicate neuron and so fits the hypothesis of the necessary conditions.
	 \hfill
\end{proof}
\begin{lemma}
		Let $ q: \C \rar \D $ be a projection. If $ \sigma \in \D $ is a maximal codeword then there exists a maximal codeword $ \tau \in \C $  such that $ q(\tau)= \sigma. $ \label{maxpro}
\end{lemma}

\begin{proof}
	Let us assume that $ q: \C\to 
	\D $ is a projection map by deleting the last ($ n^{\text{th}} $) neuron of codewords of $ \C. $ Then clearly $ q $ a is surjective map. Let $ \sigma= \sigma_1^{}\sigma_2^{}\cdots\sigma_{n-1}^{}\in \D.$  Therefore there exists $ \tau \in \C $ such that $ q(\tau)= \sigma. $  Moreover, we precisely know the choices of $ \tau. $ It can either be $ \sigma $ followed by $ 1 $ or $ 0. $ Label $ \alpha:= \sigma_1^{}\sigma_2^{}\cdots\sigma_{n-1}^{}0$ and $ \beta:=\sigma_1^{}\sigma_2^{}\cdots\sigma_{n-1}^{}1. $ Now $ \C $ may have $ \alpha$ or $ \beta, $ or both as its elements.  Clearly, $ \alpha \subseteq \beta, $ therefore the case in which both $ \alpha $ and $ \beta $ exist in $ \C $ is redundant. So, we only have the following two cases. \\
	\textbf{Case 1:} $\beta \in \C.$ In this case we claim $ \beta $ is a maximal codeword in $ \C. $ Suppose not. Then there exists $ \gamma\in \C $ such that $ \beta\subsetneq  \gamma$ then by Observation \ref{obssig} we have $ q(\beta)\subseteq q(\gamma). $ But as $ \sigma=q(\beta) $ is a maximal codeword in $ \D $ we get $ q(\beta) = q(\gamma). $ This implies $ \beta=\gamma $ or $ \alpha=\gamma. $  This is a contradiction as $ \beta\subseteq \gamma $ and $ \alpha \subseteq \beta $ and so, $ \alpha\not=\gamma. $ \\
	\textbf{Case 2:}  $ \beta^{}\not\in \C.$ In this case we claim that $ \alpha $ is maximal codeword and the proof is similar to the previous case.\\
	Hence the proof.
\end{proof}
\begin{remark}
	
	Converse of Lemma \ref{maxpro} need not hold. For example consider the code $ \linebreak\C=\{100,010,001,011,101,110\} $ and project the code to get $ \D= \{00,10,01,11\} . $ Clearly, $ 011 \in \C $ is a maximal code but $ q(011)=01 \subseteq 11. $ This implies that $ 011 $ is no more a maximal codeword after projection.
\end{remark}
	\begin{remark}
	In this remark we see binary representation of intersection of two codewords. We will use this idea in our next proof. 
	 Let $ \alpha,\beta \in \C $ be two codewords and $ \gamma= \tau_1^{}\cap \tau_2^{}. $ Let $\alpha=\alpha_1\cdots\alpha_n$, $\beta=\beta_1\cdots\beta_n$ and $\gamma=\gamma_1\cdots\gamma_n$ be their binary representation, respectively. Then we observe that binary representation of  $\gamma  $ is given as: $ \gamma_j = \begin{cases}
			1\qquad &\text{ if } \alpha_j=\beta_j=1 \\
			0\qquad & \text{ otherwise}.
		\end{cases} $
\end{remark}
\no Next we have the main result of this section. 
\begin{theorem}
	Let $ q:\C \rar\D $ be a surjective code map corresponding to a neural ring homomorphism. Then if $ \C $ is max intersection-complete, so is $ \D. $ \label{thmic}
\end{theorem}
\begin{proof}
	By Theorem \ref{thnrh} the surjective code map will be a composition of permutations, adding/ deleting a trivial or duplicate neuron, or projection. So, it is sufficient to assume all of them independently and prove the above statement. 	Let $\alpha, \beta \in \D $ be maximal codewords, we need to show that $ \alpha \cap \beta\in \D. $	
\no\textbf{\textit{Permutation:}} As $ q $ is a bijection,  there exists unique $ \sigma,\tau \in \C $ such that $ \alpha=q(\sigma), \text{ and } \beta=q(\tau). $ By Lemma \ref{maxiso}, $ \sigma,\tau \in \C $ are maximal codewords. This implies by hypothesis $\lambda= \sigma^{}\cap \tau \in \C. $ Let $\sigma=\sigma_1\cdots\sigma_n$, $\tau=\tau_1\cdots\tau_n$ and $\lambda=\lambda_1\cdots\lambda_n$ be their binary representation, respectively. Further, let $ p\in S_n$ be a permutation.   Then we have $ \alpha= \sigma_{p(1)}^{}\sigma_{p(2)}^{}\cdots\sigma_{p(n)}^{} $ and $\beta=\tau_{p(1)}^{}\tau_{p(2)}^{}\cdots\tau_{p(n)}^{}$ Then let $ q(\sigma)\cap q(\tau)=\alpha^{} \cap \beta^{}: = \gamma = \gamma_1\gamma_2\cdots\gamma_n; \text{ where }$ $$ \gamma_j= \begin{cases}
		1\qquad &\text{ if } \alpha_{j}^{}=\beta_{j}^{}=1 \\
		0\qquad & \text{ otherwise}
	\end{cases} = \begin{cases}
		1\qquad &\text{ if } \sigma_{p(j)}^{}=\tau_{p(j)}^{}=1 \\
		0\qquad & \text{ otherwise}
	\end{cases} = \lambda_{p(i)}. $$ This implies $\gamma=  \lambda_{p(1)}^{}\lambda_{p(2)}^{}\cdots\lambda_{p(n)}^{}= q(\lambda)\in \D. $
\newline

\no\textbf{\textit{Adding a trivial or duplicate neuron:}} As $ q $ is a bijection,  there exists unique $ \sigma,\tau \in \C $ such that $ \alpha=q(\sigma), \text{ and } \beta=q(\tau). $ By Lemma \ref{maxiso}, $ \sigma,\tau \in \C $ are maximal codewords. This implies by hypothesis $\lambda= \sigma^{}\cap \tau \in \C. $ Let $\sigma=\sigma_1\cdots\sigma_n$, $\tau=\tau_1\cdots\tau_n$ and $\lambda=\lambda_1\cdots\lambda_n$ be their binary representation, respectively. Then  $ \alpha= \sigma_{1}^{}\sigma_{2}^{}\cdots\sigma_{n}^{}d$ and $\beta= \tau_{1}^{}\tau_{2}^{}\cdots\tau_{n}^{}e$, where, $ d,e\in\{0,1\}  $  depending upon the map $ q. $  It is clear that $ \alpha^{} \cap \beta^{}= \lambda_{1}^{}\lambda_{2}^{}\cdots\lambda_{n}^{}f,$ where $f=\begin{cases}
	1\qquad &\text{ if } d=e=1 \\
	0\qquad & \text{ otherwise}
\end{cases}$. As $d,e$ depend on the map $q$ we get   $\alpha^{} \cap \beta^{}=q(\lambda) \in \D.$
	\newline
\textit{\textbf{Deleting a trivial or duplicate neuron:}}
	As $ q $ is a bijection,  there exists unique $ \sigma,\tau \in \C $ such that $ \alpha=q(\sigma), \text{ and } \beta=q(\tau). $ By Lemma \ref{maxiso}, $ \sigma,\tau \in \C $ are maximal codewords. This implies by hypothesis $\lambda= \sigma^{}\cap \tau \in \C. $ Let $\sigma=\sigma_1\cdots\sigma_n$, $\tau=\tau_1\cdots\tau_n$ and $\lambda=\lambda_1\cdots\lambda_n$ be their binary representation, respectively.  Then $ \alpha= \sigma_{1}^{}\sigma_{2}^{}\cdots\sigma_{n-1}^{}$ and $ \beta= \tau_{1}^{}\tau_{2}^{}\cdots\tau_{n-1}^{}$.    It is clear that $ \alpha^{} \cap \beta^{}= \lambda_{1}^{}\lambda_{2}^{}\cdots\lambda_{n-1}^{}=q(\lambda^{}) \in \D.$\newline
	\textit{\textbf{Projection:}} We just extend the idea from  deleting a trivial or duplicate neuron in view of Lemma $ \ref{maxpro}. $ That is if $ \alpha $ and $ \beta^{} $ are maximal codewords in $ \D $ there exist maximal codewords $ \sigma^{},\tau^{} \in \C $ such that $ q(\sigma)=\alpha^{} $ and $ q(\tau)=\beta. $ Rest follows.
	
	Hence the proof.
\end{proof}
\begin{remark}
	The converse of Theorem \ref{thmic} need not be true. For example consider the codes $ \C=\{100,010,001\} $ and $ \D=\{00,10,01\}. $ Consider the projection map $ q:\C \rar \D, 100\mapsto 10, 010 \mapsto 01 \text{ and } 001 \mapsto 00.  $ The map $ q $ satisfies the hypothesis of the converse. But $ \C $ is not max intersection-complete. This led us to think that converse will hold when the code map corresponds to a neural ring isomorphism.  That is in fact true and hence we have the following corollary.
\end{remark}
\begin{corollary}
	Let $ q:\C\rar \D $ be a code map corresponding to a neural ring isomorphism. Then $ \C $ is max intersection-complete if and only if $ \D $ is max intersection-complete.
\end{corollary}
\no The proof for the sufficient condition of the corollary is exactly the proof of Theorem \ref{thmic}. Further as $q$ corresponds to a neural ring isomorphism, $q$ is bijective and $q^{-1}:\D\to \C$ also corresponds to a neural ring isomorphism. So the proof of the necessary condition of the corollary comes by considering the map $q^{-1}$ instead of $q$ in Theorem \ref{thmic}. 
 
In the next section, we ask interesting questions like counting the number of possible neural ring endomorphisms for some specific class of neural rings. Thus the next section is going to be more in an algebraic and combinatorial direction. 

%% file: nre.tex
\section{Counting Neural ring endomorphisms}
  Denote $ \nrh{\C} $ to be the collection of all neural ring endomorphisms on $\ring{\C}$. Our first natural question is the structure of $ \nrh{\C} $. We observe that $ \nrh{\C} $ has a monoid structure with binary operation as the usual function composition.   The second natural question we ask is the cardinality of $ \nrh{\C} $ for a given code $\C$ on $n$ neurons. The motivation for this question has simply been to study the object $\nrh{\C}$ for a given code $\C$.  So, this section is devoted to finding the cardinality of $\nrh{\C}$ for a specific class of codes $\C$. Our specific class of interest is ``\textit{circulant codes}'' (refer to Section \ref{def}). This class is, in fact, a subclass of periodic codes introduced by Brown and Curto \cite{brown2022periodic}.  Calculating the cardinality of the entire class may be a larger and a difficult question; instead we will work on a smaller subclass ``\textit{circulant codes of support $p$}'' (refer to Section \ref{def}). We show that it is enough to work with this subclass to be able to give the answers for the larger question. To establish this sufficiency condition, we have the following observation.


\begin{obs} \label{obsnrh}
 Let $ \C' $ be a code \label{obsnrhpt} obtained from $ \C $ after applying any of the elementary code maps (1) to (3)  of Theorem \ref{thmnycc}. We observe that there is a one-one correspondence between $ \nrh{\C} $ and $ \nrh{\C'} $.  	Let $ q: \C\rar \C'$ be any of the elementary code maps (1) to (3) of Theorem \ref{thmnycc}. Then by Theorem \ref{thmnycc} we have that the corresponding neural ring homomorphism, $ \alpha_q: \ring{\C'}\to \ring{\C} $ is in fact a neural ring isomorphism. Define the correspondence as the conjugation by $\alpha_q^{-1}$, i.e.,
		\begin{align*}
			\Phi: \nrh{\C}& \rar \nrh{\C'}\\
			\phi& \mapsto \alpha_q^{-1}\circ \phi \circ \alpha_q.
		\end{align*} 
		The image of a map $ \phi \in\nrh{\C} $ under the map $ \Phi $ is defined by compositions of neural ring endomorphisms and is thus again a neural ring endomorphism. The map $ \Phi $ is a bijection with its inverse being conjugation by $ \alpha_q $. Therefore we have $ \vert \nrh{\C} \vert = \vert \nrh{\C'} \vert.$ Moreover, $ \Phi $ is a monoid isomorphism since it preserves composition and identity.

\end{obs}
\no We will save this observation for section \ref{def} specifically remark \ref{remcccp}.
\subsection{Classification of ring endomorphisms on neural codes} 
Let $ \C =\{c_1^{},c_2^{},\dots,c_m^{}\} $ be a code on $ n $ neurons and   $ c_i^{}=c_{i1}^{}c_{i2}^{}\cdots c_{in}^{} $ be the binary representation of  $ c_i, $ where $ c_{ij}^{}\in\{0,1\}. $
As discussed in the introduction, Curto et al. \cite{curto2013neural} defined the neural ring associated to a code $\C$, as $ \ring{\C}= \mathbb{F}_2[x_1,x_2,\dots,x_n]/I_\C $ where  $ I_\C=\{f\in\mathbb{F}_2[x_1,x_2,\dots,x_n] \mid f(c)=0 \text{ for all } c\in \C\}. $ The elements of $\ring{\C}$ can be expressed as polynomials, with the understanding that a polynomial is a representative of its equivalence class mod $I_\C$.  Furthermore, there is a ring isomorphism between $\ring{\C}$ and ring of functions from $\C$ to $ \mathbb{F}_2$. Note that the ring of functions from $\C $ to $ \mathbb{F}_2$ is also a vector space over $\mathbb{F}_2$. Thus  a canonical vector space structure is induced on $\ring{\C}$. The elements of the ring $\ring{\C}$ can thus be seen as functions from $\C $ to $ \mathbb{F}_2.$ For all $i\in[m],\ \rho_{c_i}:\C\to \mathbb{F}_2$ denotes the characteristic function given by, $$\rho_{c_i}(v)= \begin{cases}
	1 & \text{ if } v=c_i \\ 0 & \text{ otherwise,}
\end{cases}\quad \text{ for  any } v\in\C. $$ In polynomial notation, $$\rho_{c_i}=\prod_{c_{ij}=1}x_j\prod_{c_{ik}=o}(1-x_k).$$ Further, throughout this section we write $ \rho_{c_i} $ as just $ \rho_i $. Moreover, the set of characteristic functions $\{\rho_i\mid i\in[m]\}$ form a basis of the vector space $\ring{\C}$ over $\mathbb{F}_2.$ Therefore, $\ring{\C}$ is a $m-$ dimensional vector space over $\mathbb{F}_2$. Hence $\ring{\C}$ is isomorphic to $m$ copies of $\mathbb{F}_2$ i.e, $\ring{\C}\cong\mathbb{F}_2\underbrace{\oplus\dots\oplus}_{m-\text{times}}\mathbb{F}_2$ as a vector space over $\mathbb{F}_2.$

Denote $ \rh{\C} $ to be the collection of all ring homomorphisms that preserve unity from $ \ring{\C} $ into itself. Note that $\rh{\C}$ is a semi-group with the composition of functions as the binary operation. In 1974, Maxson \cite{maxson1974endomorphism} explored the semi-group of endomorphisms of a ring. He proved that the semi-group of endomorphisms of $ \mathbb{F}_2\underbrace{\oplus\dots\oplus}_{m-\text{times}}\mathbb{F}_2  $ is the set of all the partial functions from $ [m] $ into itself and the endomorphisms which preserve unity corresponds to all the functions from $ [m] $ into itself. Observe that the cardinality of the set of all partial functions from $ [m]  $ to itself is $ (m+1)^m $ and the cardinality of the set of all functions from $ [m] $ to itself is $ m^m. $ Therefore  $ |\rh{\C}|=m^m. $

Let us now  describe an arbitrary map $ \phi \in \rh{\C}. $ The map $\phi$  is a ring homomorphism. Moreover, $\phi$ will also be a linear map. So, to understand the map $\phi$ it sufficient to know the value of $ \phi $ on basis elements $ \{\rho_i\mid i
\in [m]\}. $ Let $ \phi $ map $ \rho_i $ to $ \sum_{j=1}^m a_{ij}^{} \rho_j^{}, $ where $ a_{ij}^{}\in\mathbb{F}_2.  $  Therefore we say that $ \phi $ is determined by these vectors $ a_i\ (\phi \leftrightarrow \{a_i\}_{i\in[m]}^{}) , $ where $ a_i^{}=(a_{i1}^{},a_{i2}^{},\dots,a_{im}^{})\in\mathbb{F}_2^m. $ Since the map $ \phi $ is a ring homomorphism, it will preserve the multiplication of $ \ring{\C}$. We will now obtain conditions on vectors $ a_i, $ so that $ \phi $ preserves multiplication.
We use the following facts given in \cite{curto2020neural}: 

$$(1)\ \rho_i \rho_j =\begin{cases}
	0 \ \ &\text{ if }  i \not =j \\
	\rho_i \ \ &\text{ if } i=j,
\end{cases}\qquad \qquad (2)\ \sum_{i=1}^{m}\rho_i=\unit.$$

\no We fix the notation $ |a_i| $ for  the number of one's occurring in $ a_i^{} $.

\begin{remark} \label{obsrh}
	In this remark we will derive some conditions on the vectors $ a_i $  defined above.
	\begin{enumerate}
		\item $ \phi(\rho_i)\phi(\rho_j) =\sum_{l=1}^m a_{il}^{} \rho_l^{} \sum_{k=1}^m a_{jk}^{} \rho_k^{}=\sum_{r=1}^m b_{ijr}^{} \rho_r^{},$
		where $ b_{ijr}^{}=a_{ir} ^{}a_{jr}^{}.$ 
		\item When $ i\not=j\in [m] $ we have $\phi(\rho_i)\phi(\rho_j)=\phi(\rho_i\rho_j)=\phi(0)=0 .$ Therefore $ \sum_{k=1}^{m} b_{ijk} \rho_k=0$. So, $ b_{ijk}=0 $ for all $ k, $ whenever $i\not=j$.  
		\item Suppose for  some $ i,k\in[m] $ let $ a_{ik}=1. $ Then for all $ j\not=i\in [m],  $  we have from observation (2)  $ 0=b_{ijk}=a_{ik}a_{jk} $. This gives $ a_{jk}=0 $.  This means for a  given  coordinate $ k\in[m] $,  we have at most one vector $ a_i $ such that $ a_{ik}=1. $ So, the number of ones in all $ a_i $'s together is at most $ m. $ Therefore \label{reml} $ \sum_{i=1}^m |a_i|\leq m. $
		\item 
		We know that $ \sum_{i=1}^{m} \rho_i=\unit = \phi(\unit)= \phi\left(\sum_{i=1}^{m} \rho_i\right)= \sum_{i=1}^{m}\phi(\rho_i)= \sum_{i=1}^{m}\sum_{j=1}^{m}a_{ij}^{}\rho_j^{}=\sum_{j=1}^{m}\sum_{i=1}^{m}a_{ij}^{}\rho_j^{} = \sum_{i=1}^m a_{i1}^{}\rho_1^{}+\sum_{i=1}^m a_{i2}^{}\rho_2^{}+\dots+\sum_{i=1}^m a_{im}^{}\rho_m^{}. $ Comparing coefficients on both sides we get  $ \sum_{i=1}^{m} a_{ik}^{}=1 $ for all $ k\in[m]. $  This means for a  given  coordinate $ k\in[m] $,  we have at least one vector $ a_i $ such that $ a_{ik}=1. $  So, the number of ones in all $ a_i $'s together is at least $ m. $ Therefore $ \sum_{i=1}^m |a_i|\geq m. $ This and observation (\ref{reml}) gives us  $ \sum_{i=1}^{m}|a_i|= m. $ \label{obspoint}
		\item   If there is a vector $ a_i $ with $ |a_i|=r. $ Then   observation (\ref{obspoint}) guarantees that there will be at least $ r-1 $ number of $ j $'s such that $ a_j $ is a zero vector. Furthermore, if we assume that there exists an $ i\in [m] $ such that $ |a_i|=m,$ i.e.,  $ a_i $ is an all ones vector, then for all $ j\not=i $ we have $ a_j $ is a zero vector.
		
	\end{enumerate}	
\end{remark}
\no We will now define three different classes of maps in $ \rh{\C}. $
\begin{definition}
	\begin{enumerate}
		\item \textbf{Basis permutation maps (BPM)}: We call an element $ \phi\in \rh{\C} $ a \textit{basis permutation map} if for all $ i\in[m],\ |a_i|=1 $.  There are $ m! $ number of such maps. We will denote $ \bpm{\C} $ as the set of all basis permutations maps from $ \ring{\C} $ into itself.
		\item \textbf{Unity maps (UM)}: We call an element $ \phi\in \rh{\C} $ a \textit{unity map} if there exists $ i\in [m] $ such that $ |a_i|=m $. From Remark \ref{obsrh} all the other vectors determining $ \phi $ will then be zero vectors. Therefore there are exactly $ m $ such maps. We will denote $ \um{\C} $ as the set of all unity maps from $ \ring{\C} $ into itself.
		\item \textbf{Non-BPM and non-UM}: These are the maps in $ \rh{\C}$ other than basis permutations and unity maps. So, cardinality of the set containing non-BPM and non-UM  is then equal to $ m^m-m!-m. $ Let $ \psi $ be a map in this class. As $ \psi $ is not a BPM there exists at least one $ i\in [m] $ such that $ |a_i|\geq 2. $ Therefore at least one other vector $ a_j $ associated to $\psi$ must be a zero vector. So, we refer to this class as non unity maps with at least one $ a_j =0.$
	\end{enumerate}
\end{definition}
Note that when $m=1,$ the cardinality of $\rh{\C}$ is exactly 1, which is the identity map.  Clearly by definitions, this identity map will be both BPM and unity map. Also, note that in this case there are no non-BPM and non-UM. When $m>1,$ it is evident from the definition of these three classes of maps that they form a classification of $\rh{\C}$. 
\begin{eg}
	Let $ \C $ be a code on $ n $ neurons with $ |\C|=3. $ We know that $ \{\rho_1^{},\rho_2^{},\rho_3^{}\} $ generates $ \ring{\C}. $ We give examples of three different ring endomorphisms one from each class on $ \ring{\C}. $
	\vspace{-0.15cm}
	\begin{enumerate}
		\item  Let $ a_1=(0,1,0),\ a_2=(0,0,1) $ and $ a_3=(1,0,0). $ The map $ \phi $ given by $ \{a_i\}_{i\in[3]} $ is a basis permutation map. Moreover, $\phi $ maps basis as follows: $ \rho_1\mapsto\rho_2,\ \rho_2\mapsto\rho_3,\ \rho_3\mapsto \rho_1. $
		\item Let $ a_1=(0,0,0),\ a_2=(1,1,1) $ and $ a_3=(0,0,0). $  The map $ \phi $ given by $ \{a_i\}_{i\in[3]} $ is a unity map. Moreover, $\phi $ maps basis as follows: $ \rho_1\mapsto 0,\ \rho_2\mapsto\rho_1+\rho_2+\rho_3= \unit,\ \rho_3\mapsto 0. $
		\item Let $ a_1=(1,0,1),\ a_2=(0,0,0) $ and $ a_3=(0,1,0). $ The map $ \phi $ given by $ \{a_i\}_{i\in[3]} $ is a non-BPM and non-UM . Moreover, $\phi $ maps basis as follows: $ \rho_1\mapsto \rho_1+\rho_3,\ \rho_2\mapsto 0,\  \rho_3\mapsto \rho_2. $
	\end{enumerate}
\end{eg}
\begin{remark}
	Let $ \phi\in\rh{\C} $ be a unity map. Recall that  $ x_j=\sum_{c_{ij}=1} \rho_i $. Then $ \phi(x_j)\in \{0,1\} $ for all $ j\in  [n].$ This is because $ \phi{(\rho_j)}\in \{0,1\} $ for all $j\in[m]$. Therefore irrespective of the code, all  \textit{unity maps} are neural ring endomorphisms. In particular, $ |\nrh{\C}| \geq m. $
\end{remark}
In the following subsection, we will restrict ourselves to codes on  $ n $ neurons with cardinality $m=n$. The rationale for this restriction is our focus on a specific class called `circulant codes with support $p$'. 
\subsection{Circulant codes}
Let $ \D=\{d_1^{},d_{2}^{},\dots,d_n^{}\} $ be any code on $ n $ neurons. For all $ i\in[n] $ let  $d_i= d_{i1}d_{i2}\cdots d_{in} $ be the binary representation of $ d_i. $ The correspondent matrix of the code $ \D $ is defined as an $ n\times n $ matrix with entries $ d_{ij}^{}. $ A circulant matrix of order $ n $ is a square matrix which has a property that each row is same as its previous row, just shifted to the right by one element, and the last element gets shifted to the first position.
 Any circulant matrix $A$ has the following general form:
 $$A=\begin{pmatrix}a_1&a_2&\cdots&a_n
 	\\a_n&a_1&\cdots&a_{n-1}\\ \vdots&\vdots&\ddots&\vdots\\a_2&a_3&\cdots &a_1\end{pmatrix}.$$
So we can observe from the general form above that one row is enough to determine the entire circulant matrix.

 Consider the codeword  $ c_1^{}=10\cdots0, $  i.e., $ 1  $ followed by $ n-1 $ zeros. Shift $ 1 $ to the right to generate the next codeword. Iterate this process and get  the remaining $ n-2 $ codewords.  In other words $ c_i $ will be a codeword containing $ 1 $ in $ i^{\text{th}} $ place and $ 0  $ elsewhere. Let $ \C=\{c_i\}_{i=1}^n $ be the code with codewords obtained as above. The correspondent matrix of the code $ \C $ is a circulant matrix. Next, consider $ c'_1=1100\cdots0 $ and similarly obtain a code $ \C' $ using the above process.  The correspondent matrix of the code $ \C' $   is also circulant. We give a generalized definition of such codes.

\begin{definition}[Circulant code] \label{def}
	A code $  \C $ on $ n $ neurons is called \textit{circulant code} if the correspondent matrix of the code $ \C $ is circulant.  
\end{definition}
\no Note that the definition automatically gives that $|\C|=n$. 
\begin{definition}[Circulant code with support $ p$]
	A code $ \C=\{c_1^{},c_2^{},\dots,c_n^{}\} $  on $n$ neurons is called \textit{circulant code with support $ p\ (1\leq p < n) $}  if $ \C $ is a circulant code and  $ c_p=11\cdots10\cdots0 $ with $ p $ consecutive ones followed by $n-p$ zeros.  \end{definition}	

\begin{remark}
Note that for a circulant code with support $ p $,  all the codewords $ c_i $ have $ |\operatorname{supp}(c_i)|=p $, where $ \operatorname{supp}(c_i) $ is the support of the codeword $ c_i=c_{i1}^{}c_{i2}^{}\dots c_{in}^{} $ which is the set  $\left\{j\in \left[n\right] \vert\ c_{ij}^{}=1\right\} $. Observe that a circulant code with support $ p $ is always a circulant code but not the other way around. For example, consider the code $ \C=\{101,110,011\} $ and $ \C'=\{110,011,101\} $ with elements reordered. Then $ \C $ is a circulant code on $ n=3  $ neurons with support $ p=2 $ whereas, $ \C' $ is no more a circulant code with support $ p=2 $.	Also, note that we do not consider $ p=n$ as in that case $ \C=\{11\cdots11\} $ is a code with cardinality $ 1. $ Furthermore, we are interested only in the codes on $ n $ neurons with cardinality $ n. $ An important point to note here is that we have fixed the order of the elements in a circulant code  with support $ p $. 

\end{remark}
\begin{eg}
	The following are few examples of \textit{circulant codes with support $ p $.} 
	\begin{enumerate}
		\item The code $ \{100,010,001\} $ is a circulant code with support $ p=1 $ on  $ n=3 $ neurons.
		\item The code $ \{1001,1100,0110,0011\} $ is a circulant code with support $ p=2 $ on  $ n=4 $ neurons.
	\end{enumerate}
\end{eg}
\begin{remark} \label{remcccp}
As mentioned in the beginning of this section, our aim is to investigate $ \nrh{\D} $ and give its cardinality for \textit{circulant codes}. To count $ \nrh{\D} $ for a circulant code $ \D $, we first convert the given code to a circulant code with some support $ p$ and label it as $ \C. $ We can do this via a permutation map. Moreover, using  Observation \ref{obsnrh} we get $ |\nrh{\D}|=|\nrh{\C}|. $ So, it is enough to work with circulant code with support  $ p. $ 
\end{remark}

A given map $ \phi\in \rh{\C} $ belongs to $ \nrh{\C} $ if for all $ i\in[n], \phi(x_i)\in \{x_i|i\in [n]\} \cup \{0,1\}.$
 So we need to understand what $ x_i $'s are in the circulant codes with support $ p $. First, we note that the number of $\rho_j$'s in the expression of $ x_i $ comes from the number of 1's in $ i^{\text{th}} $ column of the correspondent matrix of the code.  For a circulant code, the correspondent matrix is a circulant matrix. Also, in a circulant matrix, the row sum and column sum for all rows and columns is a constant. Therefore we get that in a circulant code of support $ p $,  the number of terms in $ x_i $ is the same for all $ i\in [n]. $  In fact each $ x_i $ will be a sum of $ p $ terms. Furthermore, we can see that $ x_i $ will have $ p $ consecutive terms taken circularly. For example, we observe that when $ n=6,p=4 $ we get $ x_5=\rho_5+\rho_6+\rho_1+\rho_2. $ We will now introduce a notation to write this rigorously. Given any $m\in \mathbb{Z}$ define $\mdsum{m}$ as $$ \mdsum{m}
^{}  = \begin{cases}
	n & \text{ if } m= kn, \text{ where } k\in \mathbb{Z}\\
	m\bmod n & \text{ otherwise. } 
\end{cases}$$ Note that $ \mdsum{m}=n $ when $ m $ is a multiple of $ n $ and it is $ m\bmod n $ otherwise.

\no Therefore, for $i\in[n]$ we write $$ x_i=\sum_{k=0}^{p-1} \rho_{\mdsum{i+k}}^{}. $$   Moreover, $ \rho_{i}$ and $ \rho_{\mdsum{i+p-1}} $ are respectively the first and last (or $ p^{\text{th}} $) term in the expression of $ x_i $.

Let $ \C $ be a circulant code with support $ p $ on $ n $ neurons. In the remaining part of this section we will count the number of neural ring endomorphisms of $ \C $ for $ p\in \{1,2,3,n-1\}. $ We are still working on remaining cases for $3< p<n-1 $. We have proposed conjectures for $ 3< p<n-1$ towards the end of this section. Figure \ref{fig4} summarizes this section. The upcoming result talks about the count of $ \nrh{\C} $ for circulant codes with support $ p=1 $ and $ p=n-1. $

\begin{proposition}
	If $ \C $ is a circulant code with support $ p=1 $ or $ n-1, $ then $ |\nrh{\C}|=n!+n. $ \label{propcirnrh} 
\end{proposition}
\begin{proof}
	\textbf{Case 1: }{$ p=1 $}\\ When $p=1 $ we have $ x_i=\rho_i $ for all $ i. $ Let $ \phi\in\bpm{\C}. $ Then given any $ \rho_i $ there exists some $ \rho_j  $ such that $ \phi(\rho_i)=\rho_j $ for $ i,j\in[n]. $ Therefore for all $ i\in[n] $ we have $ \phi(x_i)=\phi(\rho_i)=\rho_j=x_j $ for some $ j\in [n]. $ This implies that  $\phi\in \nrh{\C}. $ Moreover,  we already know that $ \um{\C}\subseteq \nrh{\C}  $ for any code $ \C. $ So, we have $ \bpm{\C} \cup \um{\C}\subseteq \nrh{\C} .$  It is left to show that given any non-BPM and non-UM, it is not in $ \nrh{\C}. $ Let $ \psi $ be a non-BPM and non-UM with $ \{a_i\}_{i\in[n]} $ as its representing vectors. We know that  there exists $ j\in [n] $ such that $ |a_j|=k $, where $ 2\leq k \leq n-1.$ Consider, $$ \psi(x_j)=\psi(\rho_j)=\sum_{l=1}^{n}a_{jl}^{}\rho_l=\sum_{a_{jl}^{}=1}\rho_l. $$ As $ |a_j^{}|=k $ we see that $ \psi(x_j) $ has $ k $ terms in its expression.   Since , $ 2\leq k \leq n-1$ we have $ \psi(x_j)\not \in \{x_i|i\in [n]\} \cup \{0,1\} $ . Therefore $ \psi\not \in \nrh{\C}. $ Hence  $ \bpm{\C} \cup \um{\C} =\nrh{\C} $ and the result follows.\\
	\textbf{Case 2:} $ p=n-1 $\\ When $ p=n-1 $ for $ i\in[n] $ we get $x_i=\sum_{k=0}^{n-2} \rho_{\mdsum{i+ k}} ^{}. $ Firstly observe that if $ \phi\in \bpm{\C} $ then $ \phi(x_i) $ will also have exactly $ n-1  $ terms. Secondly, the combination of $ n $ number of $ \rho_i $'s taken $ n-1 $ at a time without repetition is $ \binom{n}{n-1}=n $ choices. Further all these $ n $ choices are included in $ x_i $'s as they are exactly $ n $ distinct of them. Therefore there exists $ j\in [n] $ such that after rearrangement of terms in the expression of $ \phi(x_i) $ we get $ \phi(x_i)=x_j. $ This implies $ \phi\in\nrh{\C}. $ Hence we have $ \bpm{\C} \cup \um{\C}\subseteq \nrh{\C}. $ It is once again left to show that $ \psi, $ a non-BPM and non-UM is not in $ \nrh{\C}. $ Let $ \{a_i\}_{i\in[n]} $ be the vectors that represent $ \psi.	 $  As we noticed in Case 1, there exists $ j \in [n]$ such that $ |a_j|=k\  (2
	\leq k\leq n-1). $  Assume that there are $ r $ vectors $ \{a_{r_1},a_{r_2}\dots, a_{r_r}\} $ which take the other $ n-k $ ones, i.e., $ \sum_{i=1}^{r} |a_{r_i}^{}|=n-k$. So remaining $ n-r-1 $ vectors, say $ \{a_{t_1},a_{t_2},\dots, a_{t_{n-r-1}}\} $ are zero. From Remark \ref{obsrh} we get $ n-r-1\geq 1, $ and this implies that $ r<n-1. $ As we have mentioned earlier that all the term combinations are present in $ x_i. $ This implies there exists $ j\in [n] $ such that $ x_j= \rho_{r_1}+\rho_{r_2}+\dots+\rho_{r_r}+\rho_{t_1}+\rho_{t_2}\dots+\rho_{t_{n-r-1}}$. This implies $ \psi(x_j) $ will have $ r $ terms in its expression. As  $ 0<r < n-1 $ we have $ \psi(x_j)\not\in \{x_i\mid i\in [n]\} \cup \{0,1\}$. Therefore $ \psi \not\in \nrh{\C}. $ Hence  $ \bpm{\C} \cup \um{\C} =\nrh{\C} $ and the result follows. 
	\end{proof}
	\begin{figure}[]
		\begin{tikzpicture}[scale=5.1] 
			\draw[black] (1.2,0.79) rectangle (1.59,0.96);
			\draw (1.2,0.87)node[right] {\Small$ \Big|\nrh{\C}\Big|$};
			\draw  (0,-3pt)node[above,align=left] {\Small$ \C:$ a circulant \\ \Small code, on $n $ \\ \Small neurons  with\\ \Small support $ p $};
			\draw[black] (-0.25,-0.09) rectangle (0.24,0.3);
			\draw[black] (0.46,0.81) rectangle (0.79,0.90);
			\draw (0.63,0.9)node[below] {\Small support =$ p$};
			\draw[->,  thick, black] (0.24,0.09) -- (0.4,0.09);
			\draw[-,  thick, black] (0.4,0.7) -- (0.4,-0.8);
			\draw[->,  thick, black] (0.4,0.7) -- (0.6,0.7);
			\draw (0.6,0.7)node[right] {\Small$ 1$};
			\draw[->,  thick, black] (0.4,0.5) -- (0.6,0.5);
			\draw (0.6,0.5)node[right] {\Small$ 2$};
			\draw[->,  thick, black] (0.4,0.15) -- (0.6,0.15);
			\draw (0.6,0.15)node[right] {\Small$ 3$};
			\draw (0.61,-0.03)node[right] {$ \vdots$};
			\draw (0.61,-0.55)node[right] {$ \vdots$};
			\draw[->,  thick, black] (0.4,-0.30) -- (0.6,-0.30);
			\draw(0.6,-0.30)node[right]{\Small$ p$};
			\draw[->,  thick, black] (0.4,-0.8) -- (0.6,-0.8);
			\draw (0.6,-0.8)node[right] {\Small$ n-1$};
			
			\draw[->,  thick, black] (0.68,0.7) -- (1.18,0.7);
			\draw (1.23,0.7)node[right] {\Small$ n!+n$.};
			\draw[->,  thick, black] (0.68,0.5) -- (1.18,0.5);
			\draw (1.18,0.5)node[right] {\Small$ \begin{cases}
					3n  \ \ &n \text{ is odd and } n>1 \\36 \ \  & n=4.\\
					3n+2^2\left(\frac{n}{2}\right)!    \ \ & n \text{ is even and } n>4. 
				\end{cases}$};
			\draw[->,  thick, black] (0.68,0.15) -- (1.18,0.15);
			\draw (1.18,0.15)node[right] {\Small$\begin{cases}
				3n&   n=3d+1,\  d>1\\3n &  n=3d+2,\  d>0 \\ 270 &  n=6\\ 				15n+3^2\left( \frac{n}{3} \right)! &  n=3d ,\ d>2.
				\end{cases} $};
			\draw[-,  thick, black] (0.68,-0.30) -- (0.78,-0.30);
			\draw[-,  thick, black] (0.78,-0.30) -- (0.78,-0.19);
			\draw[->,  thick, black] (0.78,-0.19) -- (1.18,-0.19);
			\draw (0.78,-0.15)node[right]{\tiny $\GCD(p,n)=1$};
			\draw (1.18,-0.19)node[right] {\Small$ \begin{cases}
					3n \qquad \quad \quad \ \ \ \ &  n=pd+1 \\ 3n\qquad \quad \quad \ \ \ \ & n=pd+2 \\
					3n\qquad \quad \quad \ \ \ \ &  n=pd+r,\  r>2.  \\
				\end{cases}$};
			\draw[-,  thick, black] (0.78,-0.30) -- (0.78,-0.45);
			\draw[->,  thick, black] (0.78,-0.45) -- (1.18,-0.45);
			\draw (0.9,-0.4)node[right]{\Small $p | n$};
			\draw (1.23,-0.45)node[right]{$(p^2+p+3)n+p^2\left(\frac{n}{p}\right)!$};
			\draw[->,  thick, black] (0.77,-0.8) -- (1.18,-0.8);
			\draw (1.23,-0.8)node[right] {\small$ n!+n.$};
			\draw (2.25,-0.45)node[right] {\text{\Small Conjecture \ref{conj2}}};
			\draw (2.25,-0.28)node[right] {\text{\Small Conjecture \ref{conj1}}};
				\draw (2.25,-0.1)node[right] {\text{\Small Proposition \ref{propgcd1}}};
				\draw (2.25,-0.19)node[right] {\text{\Small Proposition \ref{propgcd2}}};
			\draw (2.25,-0.8)node[right] {\Small\text{Theorem \ref{propcirnrh}}};
			\draw (2.25,0.7)node[right] {\Small \text{Theorem \ref{propcirnrh}}};
			\draw (2.25,0.32)node[right] {\Small \text{Remark \ref{remnp3}}};
				\draw (2.25,0.2)node[right] {\Small \text{Remark \ref{remnp3}}};
			\draw (2.25,0.1)node[right] {\Small \text{Remark \ref{remnp3}}};
			\draw (2.25,0)node[right] {\Small \text{Theorem \ref{thnp3k}}};
		\draw (2.25,0.61)node[right] {\Small \text{Theorem \ref{thnp2}}};
		\draw (2.25,0.41)node[right] {\Small \text{Theorem \ref{thnpk2}}};
		\draw (2.25,0.51)node[right] {\Small \text{Remark \ref{rem42}}};
		\end{tikzpicture}
		\caption{The above figure represents the count of neural ring endomorphisms for a circulant code on $n $ neurons with  support $ p $, where $p\in[n-1]$.}
		\label{fig4}
	\end{figure} 
	
\begin{remark}\label{rem42}
	Consider the circulant code $ \C=\{1001,1100,0110,0011\}$ on $n=4$ neurons with support $ p=2 $. For this code we observe that $ x_1=\rho_1+\rho_2,\ x_2=\rho_2+\rho_3,\ x_3=\rho_3+\rho_4 $ and $x_4                                                                                                                                          =\rho_4+\rho_1. $ We observe that for this code there are some maps which are BPM's but not NRE's. For example, consider the map $ \phi $ given on the basis: $ \phi(\rho_1)= \rho_1, \phi(\rho_2)= \rho_3, \phi(\rho_3)= \rho_2, \phi(\rho_4)=\rho_4  $. Clearly $ \phi\in\bpm{\C} $. However, it is not a neural ring endomorphism as $ \phi(x_1)=\phi(\rho_1+\rho_2)= \rho_1+\rho_3 \notin \{x_i|i\in[4]\} \cup \{0,1\}. $  Furthermore, for this code $ \C $,  $ |\bpm{\C}|=24$. However, we only found 8 basis permutation maps that are neural ring endomorphisms. The other interesting fact is that there are some non-BPM and non-UM's which are present in $ \nrh{\C}. $ By brute force we computed that there are 24 such non-BPM and non-UM  and it gives us that $ |\nrh{\C}|=36 > 4!+4. $ Also, we  observe that the BPM's in $ \nrh{\C} $ is $ 8=2\cdot4=p\cdot n, $ for $ p=2 $ and $ n=4. $ We tried to see whether this is true for all $ n,$ and we successfully obtained the following result:  
\end{remark}
\begin{lemma}
	If $ \C $ is a circulant code with support $ p=2 $ then the total number of basis permutation maps present in $ \nrh{\C} $ is $ 2n. $ \label{prp2}
\end{lemma}

\begin{proof}
	Let $ \phi\in \bpm{\C}. $ It is enough to see the restriction of $ \phi $ to basis elements to  determine the entire map. For this reason  we now start counting where $ \phi $ can map each $ \rho_i. $ We begin with $ \rho_1 .$ As $ \rho_1^{} $ can map to $ \rho_i^{} $ for any $ i\in [n] $, we get that $ \rho_1^{} $ has $ n  $ choices. Assume that $ \rho_1^{} $ is mapped to $ \rho_j^{} $ for some $ j\in[n]. $ Since $ x_1^{}=\rho_1^{}+\rho_2^{} $ and $ \phi(x_1^{})\in \{x_i^{}\mid i\in [n]\} $
	(For a map  $\phi \in \bpm{\C}$ we have that $\ \phi(x_i^{}) $ and $ x_i $ have same number of terms, leading to $ \phi(x_i^{})\not \in\{0,1\}). $ 
	Therefore $ \phi(x_1^{})=\phi(\rho_1^{}+\rho_2^{})=\rho_j+\phi(\rho_2). $ So, for $ \phi(x_1^{})\in\{x_i^{}\mid i\in[n]\}, $ we must have $ \phi(\rho_2^{}) =\rho_{\mdsum{j+1}}^{}  $ or $ \rho_{\mdsum{j-1}}^{}. $
	Therefore $ \rho_2 $ has 2 choices when $ \rho_1 $ is fixed. On fixing $ \rho_2^{} \mapsto \rho_{\mdsum{j-1}}^{} $ we similarly get two choices for $ \rho_3^{} $ i.e., $ \rho_3\mapsto \rho_j $ or $ \rho_{\mdsum{j-2}}. $ 
	But as $ \phi(\rho_1^{})= \rho_j $ we cannot have $ \phi(\rho_3^{})=\rho_j^{}. $ Also, $ \rho_3^{} $ will still have 1 choice when $ \rho_2{}\mapsto \rho_{\mdsum{j+1}}. $ Therefore $ \rho_3^{} $ has exactly one choice when $ \rho_1^{}$ and $ \rho_2^{} $ are fixed.  So, in total we will have $ 2n  $ choices. Hence the result.
\end{proof}
\begin{remark}
	So far (Propositions \ref{propcirnrh} and Lemma \ref{prp2}) we have counted the number of basis permutation maps that are neural ring endomorphisms for a circulant code with support $ p=1,2$ and  $n-1. $ We have obtained this count to be $ n!,2n $ and $n! $ respectively. We could further see that for $p=3$, the total BPM in $\nrh{\C}$ is $2n$. The pattern remains the same as $ p $ increases, which we prove in Theorem \ref{thnrhbpm}. But we first establish a lemma that we require to prove this theorem.
\end{remark}
\begin{lemma} \label{lemrho2}
	Let $ \C $ be a circulant code with support $ p\  (2 < p < n-1) $ on $ n>2 $ neurons and let $\phi\in\bpm{\C}\cap \nrh{\C}$. If $\phi(\rho_1)=\rho_j$ for some $j\in[n]$ then $\mathbf\phi(\rho_2)\in \left\{\rho_{\mdsum{{j+1}}}^{},\rho_{\mdsum{{j-1}}}^{}\right\}.$
\end{lemma}
\begin{proof}
	Suppose not. Let $ \phi(\rho_2)= \rho_{\mdsum{{j+k}}},$ where $ k\in[n]\backslash \{1,n-1\}.$ 
	Now, $$ \phi(x_1^{})=\phi(\rho_1^{}+\rho_2^{}+\dots+\rho_p^{})=\rho_j+\rho_{\mdsum{{j+k}}}^{}+\phi(\rho_3^{})+\dots+\phi(\rho_p^{}). $$ As $  \phi $ is a neural ring endomorphism $ \phi(x_1^{})\in\{x_i^{}\mid i\in[n]\} $ i.e., there exists $ l\in[n] $ such that $ \phi(x_1)=x_l^{} $. Therefore for all $ i\in [p]\backslash[2] $ we get $ \phi(\rho_i^{})=\rho_{r_i} $ such that $ \rho_{r_i} $ is present in the expression of $ x_l^{} $ and $ r_i\not=j $ or $ \mdsum{j+k}. $ Suppose $ \rho_j $ be the first term in the expression of $ x_l $ or in other words let $ x_l=x_j. $ Consider, 
	\begin{align*}
		\phi(x_2^{})=& \phi(x_{2}^{}+\rho_1^{}-\rho_1^{}) \\&=\phi(\rho_1^{}+\rho_2^{}+ \dots+\rho_p+\rho_{p+1}^{}-\rho_{1}^{})=\phi(x_1+\rho_{p+1}-\rho_1) \\
		&=x_j+\phi(\rho_{p+1}^{})-\rho_j \\& = \rho_{\mdsum{{j+1}}}^{}+\dots+ \rho_{\mdsum{j+k}}+\dots+ \rho_{\mdsum{j+(p-1)}}^{}  + \phi(\rho_{\rho_{p+1}}).
	\end{align*} As $ \phi(x_2^{}) \in\{x_i^{}\mid i\in[n]\} $ it must be a sum of some $ p $ number of consecutive  $  \rho_i $'s. This forces  $ \phi(\rho_{p+1})= \rho_j $ or $\phi(\rho_{p+1})= \rho_{\mdsum{j+p}}. $ But the former one is not possible as $ \phi(\rho_1^{})=\rho_j. $ Therefore  $ \rho_{p+1}\mapsto \rho_{\mdsum{j+p}}. $ Next, similarly calculating we get $$ \phi(x_3^{})=\rho_{\mdsum{j+1}}^{}+\dots+\rho_{\mdsum{j+k-1}}^{}+\rho_{\mdsum{j+k+1}}^{}+\dots+\rho_{\mdsum{j+p}}^{}+\phi(\rho_{p+2}). $$    For $ \phi(x_3^{}) $ to be some $ x_m $ we would require $ \phi(\rho_{p+2})=\rho_{\mdsum{j+k}}^{} $ as  $\rho_{\mdsum{j+k}}^{} $ is the missing term in the expression of $ \phi(x_3) $. But, then we would end up getting $ \phi(\rho_{p+2})=\phi(\rho_2), $  which is a contradiction. Therefore $ x_l^{}\not=x_j. $ We would get a similar contradiction even if $ \rho_j $ was the last term in the expression of $ x_l^{}. $
	Now suppose that $ \rho_j $ is in between term in the expression of $ x_l^{}, $ i.e., let $$ x_l^{}=\rho_l^{}+\dots+\rho_{j}^{}+ \rho_{\mdsum{j+1}}^{}+\dots+\rho_{\mdsum{j+k}}^{}+\dots+\rho_{\mdsum{l+p-1}}^{}. $$ Then, $$ \phi(x_2^{})=\rho_l+\dots+ \rho_{\mdsum{j+1}}^{}+\dots+\rho_{\mdsum{j+k}}+\dots+\rho_{\mdsum{l+p-1}}^{}+\phi(\rho_{p+1}^{}). $$ This implies for $ \phi(x_2^{})\in\{x_i^{}\mid i\in[n]\}, $ we need $ \phi(\rho_{p+1}^{})=\rho_j^{}. $ But this would give us $ \phi(\rho_1^{})=\phi(\rho_{p+1}^{}) $ which is a contradiction. Hence the proof.
\end{proof}
\begin{theorem} \label{thnrhbpm}
	Let $ \C $ be a circulant code with support $ p\ (1\leq p < n) $ on $ n>2 $ neurons. The total number of basis permutation maps present in $ \nrh{\C} $ is given by $ \begin{cases}
		n! &\text{ if } p=1 \text{ and } p=n-1 \\
		2n &\text{ if } 1<p<n-1
	\end{cases}. $ 
\end{theorem}
\begin{proof}
		\textbf{Case 1: }${ p=1\text{ or } n-1.}$ In this case we get the result using Proposition \ref{propcirnrh}.\newline
		\textbf{Case 2:} $ {p=2}. $ This is Lemma \ref{prp2}.\newline
		\textbf{Case 3:} $ 2<p<n-1. $ As $ p < n-1 $ and for all $ i\in[n] $ we have $ x_i=\sum_{k=0}^{p-1}\rho_{\mdsum{i+k}}^{},  $ this gives us the following equations
		\begingroup
		\allowdisplaybreaks
	 \begin{align*}
		x_1^{}&=\rho_1^{}+\rho_2^{}+\dots+\rho_p^{},\\ x_2^{}&=\rho_2^{}+\rho_3^{}+\dots+\rho_{p+1}^{},\\ x_3^{}&=\rho_3^{}+\rho_4^{}+\dots+\rho_{p+2}^{}.
	\end{align*} \endgroup	Let $ \phi\in \bpm{\C} $ be a neural ring endomorphism. As seen in the proof of Lemma \ref{prp2}, it is enough to see the restriction of $ \phi $ to basis elements. We begin with $ \rho_1 .$ As $ \rho_1^{} $ can map to $ \rho_i^{} $ for any $ i\in [n] $, we get that $ \rho_1^{} $ has $ n  $ choices. Assume that $ \rho_1^{} $ is mapped to $ \rho_j^{} $ for some $ j\in[n]. $ By Lemma \ref{lemrho2} $ \phi $ maps $ \rho_2$ to either $ \rho_{\mdsum{j+1}}^{} $ or $ \rho_{\mdsum{j-1}}^{} $. In other words $ \phi(\rho_2^{}) $ is mapped to the basis element that is adjacent to $ \phi(\rho_1^{}).$  Similarly $ \rho_3^{} $ can have 2 possibilities, i.e., it can be mapped to basis elements that are adjacent to $\phi( \rho_2^{}) $. Fix $ \rho_2^{}\mapsto \rho_{\mdsum{j+1}^{}}, $ then $ \rho_3^{}\mapsto\rho_{\mdsum{j+2}^{}} $ or $ \rho_3^{}\mapsto\rho_{j}^{}. $ But the latter is not possible as $ \phi(\rho_1^{})=\rho_j^{}. $ Even if $ \rho_2^{}\mapsto \rho_{\mdsum{j-1}}^{}  $ we get that $ \rho_3^{} $ can only be mapped to $ \rho_{\mdsum{j-2}}^{} $ for the same reason. Therefore $ \rho_3^{} $ has only one choice to get mapped, whenever $ \phi(\rho_1^{}) $ and $ \phi(\rho_2^{}) $ are already fixed. Further for $ i\in[n]\backslash [3] $  we see $ \rho_i^{} $  has just 1 choice for to be mapped to. So, the total choices for $ \phi $ to be an neural ring endomorphism is $ n\times2\times 1\times\dots\times1=2n. $ Hence the result.
\end{proof}
We know that  $ |\nrh{\C}| =n!+n $ for  circulant codes with support $ p=1 $ and $ p=n-1 $ by Proposition \ref{propcirnrh}. Now by Theorem \ref{thnrhbpm}, we get that $ |\nrh{\C}| \geq 3n $  for all circulant codes with support $p$ on  $n>2 $ neurons. Further, we investigate how non basis permutation and non unity maps behave on circulant codes with support $ 1<p<n-1. $  Before that we will introduce some notations. Let $ y_{i}=   \rho_{i_1^{}}^{}+\rho_{i_2^{}}^{}+\dots+\rho_{i_k^{}}^{}$  be some combination of $ k $ number of $ \rho_j $'s.  We will use $ \norm{y_i} $ as the notation to indicate the number of distinct $ \rho_j $'s in the expression of $ y_i. $ Therefore $ \norm{y_i}=k$ for the above expression of $ y_i $. Similarly, $\norm{x_i}=p $ for a circulant code of support $ p $ since $ x_i=\sum_{k=0}^{p-1} \rho_{\mdsum{i+k}}^{}. $ We  already know by definition that any $ \phi\in \rh{\C} $ is in $ \nrh{\C} $  if for all $ i\in [n] $ we have $ \phi(x_i)\in\{x_j\mid j\in [n]\} \cup\{0,1\}. $ With the notation $ \norm{.} $ the necessary condition for $ \phi\in
\rh{\C}$ to be in  $\nrh{\C} $ is: for all $ i\in[n] $ we must have $ \norm{\phi(x_i)}\in\{0,n,\norm{x_j}\}$  for some $ j\in[n]. $ Further, the necessary condition for a map $ \phi \in\rh{\C}$ on circulant codes $ \C $ with support $ p $ to be in $ \nrh{\C} $ will be $ \norm{\phi(x_i)}\in\{0,n,p\} $ for all $ i\in [n]$.  Note that for all $ i\in [n] $ we have $ \phi(\rho_i)=\sum_{j=1}^{n}a_{ij}^{}\rho_j= \sum_{a_{ij}=1}\rho_j^{}. $ Thus $ \norm{\phi(\rho_i^{})}=|a_i|. $    Also,  $ \norm{\phi(x_i)}=\displaystyle\sum_{k=0}^{p-1}\Big \Vert{\phi\left( \rho_{\mdsum{i+k}}^{}\right)}\Big\Vert = \sum_{k=0}^{p-1} \Big|a_{\mdsum{i+k}}^{}\Big|.$ \\ \no  We have already seen that $ |\nrh{\C}|\geq 3n $ as it consists of 2n basis permutation maps (Refer Lemma \ref{prp2}) and $ n $ unity maps. In the next theorem we will show that $  |\nrh{\C}|= 3n $  for circulant code  with support $ p=2 $ when $ n $ is odd. But we first establish a lemma that we require to prove this theorem. 
\begin{lemma}
	Let $ \C $ be a circulant code with support $ p=2 $ and let $ \phi $ be a non-BPM and non-UM with $ \{a_i\}_{i\in[n]} $ as the corresponding vectors of $\phi$. Suppose $ \phi\in\nrh{\C} $ then $|a_i|\in\{0,2\}$ for all $i\in[n].$ \label{lembpm2}
\end{lemma}
\begin{proof}
	First we claim that $ {\text{for all }  i, \text{ we have } |a_i|\leq 2. } $ Suppose not. Then there exists $ j $ such that $ |a_j|=k>2. $ Also as $ \phi $ is a non unity map we have $ k<n. $ We know that $ x_j=\rho_j+\rho_{\mdsum{{j+1}^{}}}^{}. $ By the necessary condition for $ \phi\in \nrh{\C}, $ we have that $ \norm{\phi(x_j)}=0,2 $ or $ n. $ But $ \norm{\phi(\rho_j)}=|a_j|=k>2.$ So, the only possibility is that $ \norm{\phi(x_j)}=n. $ Therefore  $ |a_{\mdsum{{j+1}}}^{}|=\norm{\phi(\rho_{\mdsum{{j+1}}^{}})}=n-k. $ Also, as $ |a_j|+\Big|a_{\mdsum{{j+1}}}^{}\Big|=n, $ we get that $ |a_i|=0, $ for all $ a_i\not=a_j $ and $ a_i\not=a_{\mdsum{{j+1}}}. $ As $ |a_{\mdsum{j-1}}|=0, $ we have $ \phi\left(\rho_{\mdsum{j-1}}^{}\right) =0.$ So, \begin{align*}
		\phi \left( x_{\mdsum{{j-1}^{}}}^{}\right)  &= \phi\left(\rho_{\mdsum{j-1}}^{}+\rho
		_j\right)=\phi\left(\rho_{\mdsum{j-1}}^{}\right)+\phi(\rho_j)= \phi\left( {\rho_j}\right) .
	\end{align*}$  $ Therefore $ \Big\Vert{\phi\left( x_{\mdsum{{j-1}^{}}}\right) }\Big\Vert=\norm{\phi(\rho_j)}=k\not=0,2  $ or $ n $ as $ 2<k<n.  $ This is a contradiction to the necessary condition of $ \phi\in\nrh{\C}. $  Hence the claim.
	
	Further we show for all $i\in[n],   |a_i|\neq=1$ 	Suppose,  there exists $ j\in[n] $ such that $ |a_j|=1. $  As $|a_j|= \norm{\phi(\rho_j)}=1, $ gives us $ \norm{\phi(x_j)} \not=0.$ Also,  for all $i\in [n], \ |a_i|\leq 2 $ so we have $ \norm{\phi(x_j^{})}=|a_j|+|a_{\mdsum{j+1}^{}}|=1+|a_{\mdsum{j+1}^{}}|\leq 1+2=3. $ Therefore $ \norm{\phi(x_j)} \not=n, $ since $ n>3. $  Thus the necessary condition gives us that $ \norm{\phi(x_j)}=2$. So, $ \norm{\phi(\rho_{\mdsum{{j+1}}}^{})}=1. $ Iteratively, for all $ i\in[n] $ that $ \norm{\phi(\rho_i)}=1=|a_i|. $ This implies that $ \phi\in\bpm{\C}, $ which is a contradiction. Hence the proof.	
\end{proof}
\begin{theorem}
	Let $ \C $ be a circulant code with support $ p=2. $ If $ n $ is odd then $ |\nrh{\C}|=3n.  $ \label{thnp2}
\end{theorem}
\vspace{-0.5cm}
\begin{proof}
	\no Clearly $ n $ cannot be 1 as $ p=2. $ \\ 
	\textbf{Case 1:} $ n=3 $ \\ As $ p=2=n-1 $ in this case. By Proposition $ \ref{propcirnrh} $ we already know that $ \nrh{\C}=n!+n=3!+3=3\cdot n. $ Hence the proof.\\
	\textbf{Case 2:} $ n>3  $\\  In this case, we are only left to show that there are no more neural ring endomorphisms.
	
	Let $ \phi $ be a non-BPM and non-UM with $ \{a_i\}_{i\in[n]} $ as the vectors that represent it. Suppose $ \phi $ is a neural ring endomorphism. By Lemma \ref{lembpm2} we have that $|a_i|\in\{0,2\}$ for all $i\in[n].$	So, $\sum_{i=1}^{n}|a_i|$ is an even number. 
	From Remark \ref{obsrh}, $ \sum_{i=1}^{n}|a_i|=n. $ This forces $ n $ to be even. This is a contradiction to the hypothesis that $  n $ is odd. Therefore $ \phi \not \in \nrh{\C}. $ Hence the proof. 
\end{proof}
In the view of Theorem \ref{thnp2} we  further count the cardinality of $ \nrh{\C} $ when $ n $ is even. In Remark \ref{rem42} we have seen that for  of a circulant code $ \C $ with support $ p=2 $ on $ n=4 $ neurons we have $ |\nrh{\C}|=36. $   We will now look for $ n\geq 6 $ in the following theorem.
\begin{theorem}
	Let $ \C $ be a circulant code  with support $ p=2. $ If $ n>4 $ is even then $ |\nrh{\C}|=3n+2^2\left(\dfrac{n}{2}\right)!. $\label{thnpk2}
\end{theorem}
\begin{proof}
Let $ n=2k $ for some $ k>2. $	We first count the total number of non-BPM and non-UM  that are in $ \nrh{\C}. $ Let $ \phi $ be a non-BPM and non-UM  with $ \{a_i\}_{i\in[n]} $ as its representing vectors. By Lemma \ref{lembpm2} for all $ i\in[n]  $ we have $ |a_i|\in\{0,2\} $   Suppose if $ |a_i|=2=|a_{\mdsum{{i+1}}}|, $ then $ \norm{\phi(x_i)}=4. $ This contradicts the necessary condition of neural ring endomorphism as $ n>4. $ This implies no two consecutive $ a_i $'s have the same value, i.e., $ |a_i|\not=|a_{\mdsum{i+1}}| $ for any $ i\in[n].  $ Thus  if $ |a_1|=2 $ then  for all $ m\in[k]   $ we get $ |a_{2m-1}|=2 $ and $ |a_{2m}|=0. $ Similarly, if $ |a_2|=2 $ for all $m\in [k] $ we get $ |a_{2m}|= 2$ and $ |a_{2m-1}|=0. $ Therefore when $ \phi \in \nrh{\C} $ there are broadly two types of choices for the vectors that can represent it. Let us fix one type of choice and count how many such neural ring endomorphisms it corresponds to. By the choice of all $ |a_i| $ we see that for all $i\in [n],\ \norm{\phi(x_i)}=2. $ This implies for all $ i\in [n] $ there exists $ j\in[n] $ such that $ \phi(x_i^{})=x_j^{}. $
	
	Assume $ |a_1|=2. $  Consider, $$\phi(x_1^{})=\phi(\rho_1^{}+\rho_2^{})=\sum_{j=1}^{n}a_{1j}\rho_j+\sum_{j=1}^{n}a_{2j}\rho_j=\sum_{j=1}^{n}a_{1j}\rho_j=\phi(\rho_1^{}). $$ Let $ \phi(x_1^{})=x_i^{} $ (say) for some $ i\in [n]. $ Then $ \phi(\rho_{1}^{})=x_i^{} $ and clearly $ \rho_1 $ has $ n $ choices. Similarly, whenever $ |a_l|=2 $ we get that $ \rho_l^{}\mapsto x_j=\rho_j^{}+\rho_{\mdsum{{j+1}}}^{}. $ In general, $ \phi $ maps every basis element to $ 0 $ or a consecutive\footnote{We consider $ \rho_n^{}+\rho_1^{} $ as a consecutive sum} sum of basis elements. In this case  as $ |a_{2m-1}|=2 $ and $ |a_{2m}|=0 $ for all $ m\in [k] $  we have  $ \phi(\rho_{2m}^{}) =0$ for all $ m\in[k]. $ So, we need to only figure out $ \phi(\rho_{2m-1}^{}).$ As we have already fixed when $ m=1 $, we  look at $ m=3, $ i.e., we need to find where $ \rho_3^{} $ is mapped by $ \phi $. Let, if possible $ \rho_3{} \mapsto x_{\mdsum{{i+r}}}^{}  $ where $0<r<n $ and $  r$ is odd. Firstly, note that as $ \phi(\rho_1^{})=x_i=\rho_i+\rho_{\mdsum{i+1}} $ and $ x_{\mdsum{i+1}}=\rho_{\mdsum{i+1}} +\rho_{\mdsum{i+2}},\ x_{\mdsum{i+(n-1)}}=\rho_{\mdsum{i+(n-1)}}+\rho_i$ we have $ r\notin \{1,n-1\} $. So now as $ r\geq 3 $ we observe that the number of $ \rho_j^{} $'s that are in between $ \rho_{\mdsum{i+1}} $ and $ \rho_{\mdsum{{i+r}}} $ is $ r-2. $      
	Note that once the $ \phi(\rho_{2m-1}) $ is chosen for all $ m\in[k-1]\backslash[2] $ there will still be one $ \rho_l $ in between $ \rho_{\mdsum{i+1}} $ and $ \rho_{\mdsum{{i+r}}} $ as $ r-2 $ is odd. In other words this process will exhaust all the sum of  consecutive basis. Now we have to map $ \rho_{n-1} $ as $ |a_{n-1}|=2. $ But there is no more sum of consecutive basis left, meaning there is no choice for  $ \phi(\rho_{n-1}^{}). $  Therefore $ \rho_3^{} $ cannot map to $ x_{\mdsum{{i+r}}}^{} $ when $  r $ is odd.  Thus $ \phi:\rho_3^{}\mapsto x_{\mdsum{i+r}} $ for some even $ r\geq 2. $ This clearly gives $ \frac{n}{2}-1 =k-1$ choices for $\rho_3^{}  $ to be mapped by $ \phi. $  Similarly we observe that $ \rho_5^{} $ will have $ k-2 $ choices. At the end we  see that $ \rho_{n-1} $ has only 1 choice. Thus in total we get $ n(k-1)! $ as the number of possible $ \phi  $ that can neural ring endomorphism when $ |a_1|=2 $.
	
	Similarly, we get $ n(k-1)! $ as the number of possible $ \phi  $ that can neural ring endomorphism when $ |a_2|=2. $ Therefore total number of non-BPM and non-UM  that are in $\nrh{\C}  $ is $ 2n(k-1)!=2n\left(\frac{n}{2}-1\right)!=2^2\left(\frac{n}{2}\right)!. $ By Lemma $ \ref{prp2} $ we already know the count of BPM that are in $ \nrh{\C} $ to be $ 2n. $ Finally adding the $ n $ unity maps we get the result. 
\end{proof}
\no	Combining the results of Theorem \ref{thnp2} and \ref{thnpk2} together, we have  $$ |\nrh{\C}|=\begin{cases}
	3n  \qquad &\text{ if } n \text{ is odd and } n>1\\
	3n+2^2\left(\dfrac{n}{2}\right)!  \qquad &\text{ if } n \text{ is even and } n>4,
\end{cases}$$ where $ \C $ is a circulant code with support $ p=2. $

Theorem \ref{thnp2} and \ref{thnpk2} gave us a hint that $ \GCD(p,n) $ could play a vital role in deciding the count of $ \nrh{\C}. $ With brute force we found that for a circulant code with support $ p=3 $ on $ n=3k+1 $ and $ n=3k+2, $  the non-BPM and non-UM  that are in $ \nrh{\C} $ is zero. This leads us to think that the non-BPM and non-UM  in $ \nrh{\C} $ is zero when $ \GCD(p,n)=1. $ We show this statement has an affirmative answer (Proposition \ref{propgcd1} and Proposition \ref{propgcd2}). But we first prove a couple of lemmas that we require in the proof of the above statement. We will also introduce some new notations to help us simplify these propositions' proof.
\begin{lemma}
	Let $ \C $ be a circulant code with support $ p>1$ with $ \GCD(p,n)=1$ and $ \phi\in\rh{\C}$ be a non-BPM and non-UM. For all $ i\in [n] $ if $ \norm{\phi(x_i)}\in\{0,p,n\} $  then $ \norm{\phi(x_i)}\not=n$. \label{lemncnrh1}
\end{lemma} 
\begin{proof}
		Let $ \phi $ be a non-BPM and non-UM such that for all $ i\in[n],\ \norm{\phi(x_i)}\in\{0,p,n\}. $ 
	Let $ \{a_i\}_{i\in[n]} $ be the vectors that represent $ \phi. $ Suppose there exists $ j\in[n] $ such that $ \norm{\phi(x_j)}=n. $ Without loss of generality let us assume $ j=1. $ As $ x_1=\rho_1+\dots+\rho_p $ we get $ n=\norm{\phi(x_1^{})}=|a_1^{}|+|a_2^{}|+\dots+|a_p^{}|. $ This implies for all  $ k\in[n]\backslash[p] $ we have $ |a_k^{}|=0. $ Let $ l\in[p] $ be the smallest  such that $ |a_l^{}|\not=0  $. Hence $ n=\norm{\phi(x_1)}=|a_l^{}|+\dots+|a_p^{}|. $ \\Consider,
	\begingroup
	\allowdisplaybreaks
	\begin{align*}
		\norm{\phi(x_1^{})}-\norm{\phi(x_l^{})}=&\Big( |a_1^{}|+\dots+|a_{l}^{}| +\dots+ |a_{p}^{}|\Big) -\left( |a_{l}^{}|+\dots+|a_p^{}|+|a_{\mdsum{p+1}}^{}|+\dots+|a_{\mdsum{l+p-1}}^{}|\right)  \\
		=& |a_1^{}|+\dots+|a_{l-1}^{}|-(|a_{\mdsum{p+1}}^{}|+\dots+|a_{\mdsum{l+p-1}}^{}|)\\
		=& |a_1^{}|+\dots+|a_{l-1}^{}| \qquad (\text{Since } |a_k^{}|=0  \text{ for all } k\in[n]\backslash[p] ) \\
		=& 0 \qquad (\text{Since } l \text{ is the smallest integer such that } |a_l^{}|\not=0 ).
	\end{align*} 
	\endgroup So, we get $ \norm{\phi(x_l^{})}=n. $ Next,  we have $ |a_{l+1}^{}|+\dots+|a_p^{}|<n  $ as $ |a_l^{}|\not=0. $ Moreover,  $ 0<|a_{l+1}^{}|+\dots+|a_p^{}|<n, $ if not $ |a_l^{}|=n $ and that is not possible as $ \phi $ is not a unity map. Consider, 
	\begin{align*}
		\norm{\phi(x_{l+1}^{})}=&|a_{l+1}^{}|+\dots+|a_p^{}|+ \dots+|a_{\mdsum{l+p}}^{}|\\
		=&|a_{l+1}^{}|+\dots+|a_p^{}| \qquad (\text{Since } |a_k^{}|=0  \text{ for all } k\in[n]\backslash[p] )\\
		\implies& 0<	\norm{\phi(x_{l+1}^{})}<n \quad (\text{Since } 0<|a_{l+1}^{}|+\dots+|a_p^{}|<n)
	\end{align*}
	Also, by the hypothesis $ \norm{\phi(x_{l+1^{}})}\in\{0,p,n\}. $  Hence $ \norm{\phi(x_{\mdsum{l+1}^{}})}=p $  and $ n-p=\norm{\phi(x_l^{})}-\norm{\phi(x_{\mdsum{l+1}^{}})}=|a_l|-|a_{ \mdsum{l+p}}|. $ Now, if   $ \mdsum{l+p}\in [n]\backslash[p], $ then   $|a_{ \mdsum{l+p}}|=0 $. Or if $ \mdsum{l+p}\in [p] $ we observe that $ \mdsum{l+p}=l+p-n<l $ as $ p<n. $ Thus if $ |a_{ \mdsum{l+p}}|\not=0,  $ it contradicts the minimality of $ l. $ Both the cases results in $ |a_{\mdsum{l+p}}|=0 $ and this implies $ |a_l|=n-p. $ Let $ m\in[p]\backslash[l] $ be the smallest such that $ |a_m|\not=0. $  Note that as $ |a_l|=n-p  $ and $ \sum_{i=1}^{p}|a_i|=n $ we get $ 0<|a_m|\leq p. $  Suppose $ |a_m|=k <p $ then \begin{align*}
		\norm{\phi(x_{\mdsum{m+1}}{})}=&|a_{m+1}|+\dots+|a_p|+\dots+|a_{\mdsum{m+p}}|\\=&n- \sum_{i=1}^{m}|a_i| \\=&n-(n-p+k)=p-k\not\in\{0,p,n\}.
	\end{align*}Therefore it ensures $ |a_m|=p. $ Thus $ |a_i|=0 $ for all $ i\in[n]\backslash\{l,m\}. $ \\
	Note that, \begin{align*}
		&	x_{\mdsum{m+n-p}}^{} = \rho_{\mdsum{m+n-p}} +\dots+\rho_1+\dots+\rho_l+\dots+\rho_{\mdsum{m+n-1}} \\  \implies & \norm{\phi(x_{\mdsum{m+n-p}}^{})}=|a_{\mdsum{m+n-p}}|+\dots+|a_l|+\dots+|a_{\mdsum{m+n-1}}|=|a_l|=n-p.
	\end{align*}  Also, for $ \norm{\phi(x_{\mdsum{m+n-p})}}\in\{0,p,n\} $ we must have $ n=p $ or $2p, $ or $ p=0 $.  But as $ \GCD(p,n)=1 $ and $ p>1 $ none of them is possible. Therefore it is a contradiction to the hypothesis. Hence $ \norm{\phi(x_i)}\not=n $ for any $ i\in[n]. $  
\end{proof}
\begin{lemma}
	Let $ \C $ be a circulant code with support $ p>1$ with $ \GCD(p,n)=1$ and $ \phi\in\rh{\C}$ be a non-BPM and non-UM. For all $ i\in [n] $ if $ \norm{\phi(x_i)}\in\{0,p,n\} $  then $ \norm{\phi(x_i)}=p $. \label{lemncnrh}
\end{lemma} 
\begin{proof}
		Let $ \phi $ be a non-BPM and non-UM such that for all $ i\in[n],\ \norm{\phi(x_i)}\in\{0,p,n\}. $ 
	Let $ \{a_i\}_{i\in[n]} $ be the vectors that represent $ \phi. $ By Lemma \ref{lemncnrh1} we already know that $\norm{\phi(x_i)}\not =n.$ It is only left to show that 0 is also not possible.
	Let if possible there exists $ j\in[n] $ such that $ \norm{\phi(x_j)}=0. $ In fact, there exists  $ k\in[n-1] $ such that $ \norm{\phi(x_{\mdsum{j+k}})}\not=0. $ Thus $ \norm{\phi(x_{\mdsum{j+k}})}=p, $ as it cannot be $ n $ using Lemma \ref{lemncnrh1}. Choose the smallest $ k $ such that $  \norm{\phi(x_{\mdsum{j+k}})}=p, $ i.e., $ \norm{\phi(x_{\mdsum{j+m}})}=0 $ for all $ m<k. $ Also, as $ x_{\mdsum{j+k-1}} = \sum_{m=0}^{p-1} \rho_{\mdsum{j+k-1+m}} $ we have $0=\norm{\phi(x_{\mdsum{j+k-1}})}=\sum_{m=0}^{p-1}|a_{\mdsum{j+k-1+m}}|.$ Therefore $ |a_{\mdsum{j+k-1+m}}|=0 $ for all $ m\in \{0\}\cup[p-1]. $ Consider $$ x_{\mdsum{j+k}}=\rho_{\mdsum{j+k}}+\dots+\rho_{\mdsum{j+k+p-1}}= x_{\mdsum{j+k-1}}-\rho_{\mdsum{j+k-1}}+\rho_{\mdsum{j+k+p-1}}. $$ So, $ p=\norm{\phi({x_{\mdsum{j+k}}})}=\norm{\phi(x_{\mdsum{j+k-1}})}- |a_{\mdsum{j+k-1}}| + |a_{\mdsum{j+k+p-1}}|=|a_{\mdsum{j+k+p-1}}|$. Next, we choose the smallest $ l>0 $ such that $ \norm{\phi(x_{\mdsum{j+k+l}})}=p $ and repeating the process as above we get $ |a_{\mdsum{j+k+l+p-1}}|=p $ and other $ |a_i| $'s corresponding to $ \rho_i $'s that are in the expression of $ x_{\mdsum{j+k+l}} $ are 0. Therefore for all $ i\in[n] $ we get  $ |a_i|\in\{0,p\}. $ As  $ \sum_{i=0}^{n}|a_i|=n $ and $ \sum_{i=0}^{n}|a_i|=dp, $ this implies $ p|n $ and $ \GCD(p,n)=p\not =1. $ This is a contradiction to our hypothesis that $ \GCD(p,n)=1. $   Hence the result. 
\end{proof}
\no In other words Lemma \ref{lemncnrh} says that  if  $ \phi $ a non-BPM and non-UM satisfies the necessary condition to be a neural ring endomorphism then $ \norm{\phi(x_i)}=p $ for all $ i\in[n]. $

\begin{obs}\label{obsgcd}
	Consider $ \C $ to be a circulant code with $ p>1 $ and $ \GCD(n,p)=1 $ with $ n=pd+r, $ where $ 0<r<p. $  Let $ \phi\in \rh{\C} $ be a non-BPM and non-UM such that for all $ i, \norm{\phi(x_i)}\in\{0,p,n\}  $. Then by Lemma \ref{lemncnrh} for all $ i\in[n] $ we get $ \norm{\phi(x_i)}=p. $ Let $ \{a_i\}_{i\in[n]} $ be  the vectors that represent $ \phi $.  In this observation we organize these vectors into batches of $ p $'s. Then we relabel the set $ \{a_i\}_{i\in[n]}  $ to write them as  $ \{\underbrace{\beta_{11},\dots,\beta_{1p}},\underbrace{\beta_{21}\dots,\beta_{2p}^{}}, \dots, \underbrace{\beta_{d1}\dots\beta_{dp}}, \beta_{(d+1) 1},\dots \beta_{(d+1) r}\},$ where   $  \beta_{ij}=a_{(i-1)p+j}^{} $ for $ i\in[d], j\in[p] $ and   $  \beta_{(d+1)j}=a_{dp+j}^{} $ for all $ j\in[r] $. Considering the vectors $\{\beta_{ij}\}$'s instead of $\{a_k\}$'s will help us simplify writing the proofs of the next two results. We will now observe some facts about $ \beta_{ij} $'s and use these facts directly in the proofs.  
	\begin{enumerate}
		\item $\sum_{j=1}^{p}|\beta_{1j}|=\sum_{j=1}^{p}|a_j|=\norm{\phi(x_1)}=p$.
		\item Similarly, for all $ i\in [d],\  \sum_{j=1}^{p}|\beta_{ij}|=p. $
		\item Note that, $\norm{\phi(x_2)}= |a_2|+\dots+|a_{p+1}|=|\beta_{12}|+\dots+|\beta_{1p}|+|\beta_{21}|.$ \item Since, $ \norm{\phi(x_1)}=\norm{\phi(x_2)}=p $ and $ \norm{\phi(x_1)}-\norm{\phi(x_2)}=|\beta_{11}|-|\beta_{21}| $. Hence, $ |\beta_{11}|=|\beta_{21}|. $
		\item Similarly using $ \norm{\phi(x_1)}=\norm{\phi(x_3)}=p $ we get $ |\beta_{12}|=|\beta_{22}|. $
		\item  Extending the above observation,  for all $ j\in[p]  $ we get $ |\beta_{1j}|=|\beta_{2j}| $.
		\item Similarly, for all $i\in[d]$ we get $ |\beta_{11}|=|\beta_{i1}|$.
		\item Furthermore, for all $i\in [d]$ and $j\in[p]$ that $|\beta_{1j}|=|\beta_{ij}| $
		\item Also, when $ i=d+1, \ |\beta_{1j}|=|\beta_{(d+1)j}| $ for all $ j\in[r]. $
		\item Consider,
		\begin{align*}
		\sum_{i=1}^{n}|a_i|	&=\sum_{i=1}^{d}\sum_{j=1}^{p}|\beta_{ij}|+\sum_{j=1}^{r}|\beta_{(d+1)j}|=pd+\sum_{j=1}^{r}|\beta_{(d+1)j}|, \\ & \text{and } 	\sum_{i=1}^{n}|a_i| =n \implies \sum_{j=1}^{r}|\beta_{(d+1)j}|=n-pd=r.  
		\end{align*}
	\item Thus,  $ \sum_{j=1}^{r}|\beta_{ij}|=r $ for all $ i\in[d+1]. $ \qquad (Using 5, 6 and 7)
	\item Since, $ \norm{\phi(x_1)}=\norm{\phi(x_n)}=p $ and $ \norm{\phi(x_n)}-\norm{\phi(x_1)}=|\beta_{(d+1)r}|-|\beta_{1p}| $. Hence, $ |\beta_{(d+1)r}|=|\beta_{1p}|. $
	\item Similarly using $ \norm{\phi(x_1)}=\norm{\phi(x_{n-j})}=p $ for all $ j\in{0}\cup [r-1] $ we get $ |\beta_{(d+1)(r-j)}|=|\beta_{1(p-j)}|. $
	\end{enumerate}   
\end{obs}

\no In the next two propositions we will show that the count of $ \nrh{\C} $ is 3n for any circulant code $\C$ with support $ 2<p<n-1 $ with $ \GCD(p,n)=1 $ for $ n=pd+1$ and $pd+2. $

\begin{proposition}\label{propgcd1}
	Let $ \C $ be a  circulant code  with support $ 2<p<n-1 $. If $ \GCD(p,n)=1 $ and $ n=pd+1 $ then $ |\nrh{\C}|=3n.$ 
\end{proposition}
\begin{proof}
	Let  $ \phi\in\rh{\C} $ be a non-BPM and non-UM  with $ \{a_{i}\}_{i\in [n]} $ as its representing vectors. Label the vectors $ \{a_i\}_{i\in[n]} $ as in Observation \ref{obsgcd} and rewrite them as $ \{\beta_{ij}\}_{i\in[d],j\in[p]} \cup\{\beta_{(d+1)j}\}_{j\in[r]}. $ Let if possible $ \phi\in\nrh{\C}. $ Then by Lemma \ref{lemncnrh} for all $ i\in[n] $ we get $ \norm{\phi(x_i)}=p. $ Now using Observation \ref{obsgcd} we will list some facts on $\beta_{ij}$'s which help us prove this theorem,
	 \begin{enumerate}
		\item $ \sum_{j=1}^{r}|\beta_{(d+1)j}|=r, $ and as $ r=1 $ we have $ |\beta_{(d+1)1}|=1. $
		\item For all $ i\in[d] $ we have $ |\beta_{i1}|=|\beta_{(d+1)1}|=1. $
		\item Also, for all $ i\in[d] $ we have $ |\beta_{ip}|=|\beta_{(d+1)1}|=1. $
		\item \label{obspo}	Consider, 
		\begin{align*}
			p&=\norm{\phi(x_n)}=|\beta_{(d+1)1}|+|\beta_{11}|+\dots+|\beta_{1(p-1)}|\\& =1+1+\sum_{j=2}^{p-1}|\beta_{1j}| \implies \sum_{j=2}^{p-1}|\beta_{1j}|=p-2.
		\end{align*} 
		\item  \label{obsp2} Also, \begin{align*}
			p&=\norm{\phi(x_{n-1})}=|\beta_{dp}|+|\beta_{(d+1)1}|+|\beta_{11}|+\sum_{j=2}^{p-1}|\beta_{1j}|-|\beta_{1(p-1)}|\\&=1+1+1+p-2-|\beta_{1(p-1)}| \implies  |\beta_{1(p-1)}|=1. \end{align*}
			\item \label{obsp3}Further, from Observation \ref{obsgcd} $ |\beta_{1{(p-1)}}|=|\beta_{i{(p-1)}}|, $ for all $ i\in[d]. $ Thus $|\beta_{i{(p-1)}}|=1$. 
			\item Similar to the discussion done for $\norm{\phi(x_n)}$ in point (\ref{obspo}), we repeat it for $\norm{\phi(x_{n-1})}$ to get $\sum_{j=2}^{p-2}|\beta_{1j}|=p-3. $ Repeating the calculations done in point (\ref{obsp2}), now for $\phi(x_{n-2}),$ and using $\sum_{j=2}^{p-2}|\beta_{1j}|=p-3$, we get $|\beta_{1(p-2)}|=1.$
			\item Similar to point (\ref{obsp3}), we get $ |\beta_{i(p-2)}|=1$ for all $i\in[d]$.
			\item Iteratively, we will get $|\beta_{i(p-j)}|=1$ for all $j\in[p-1].$    
	\end{enumerate}
	From the above points we get that all  $|\beta_{ij}|$'s are one. Since we obtained the $\beta_{ij}$'s after re-labeling the vectors $a_k$'s, so automatically $|a_k|=1$ for all $k\in[n].$ But this is a contradiction to the fact that $ \phi $ is a non-BPM and non-UM. Therefore $ \phi\not\in\nrh{\C}. $ This implies that none of the non-BPM and non-UM  are in $ \nrh{\C}. $ Moreover, by Theorem \ref{thnrhbpm} we already know the count of BPM that are in $ \nrh{\C} $ is $ 2n. $ Finally adding the $ n $ unity maps we get the result.
\end{proof}
\begin{proposition}\label{propgcd2}
	Let $ \C $ be a circulant code  with support $ 2<p<n-1 $. If $ \GCD(p,n)=1 $ and $ n=pd+2 $ then $ |\nrh{\C}|=3n.$ 
\end{proposition}
\begin{proof}
	Let  $ \phi\in\rh{\C} $ be a non-BPM and non-UM  with $ \{a_{i}\}_{i\in [n]} $ as its representing vectors. Label the vectors $ \{a_i\}_{i\in[n]} $ as in Observation \ref{obsgcd} and rewrite them as $ \{\beta_{ij}\}_{i\in[d],j\in[p]} \cup\{\beta_{(d+1)j}\}_{j\in[r]}. $ Let if possible $ \phi\in\nrh{\C}. $ Then by Lemma \ref{lemncnrh} for $ i\in[n] $ we get $ \norm{\phi(x_i)}=p. $ By Observation \ref{obsgcd},
	\begin{equation}\label{pointos3} 
		|\beta_{ip}|=|\beta_{(d+1)2}| \text{ and } |\beta_{i(p-1)}|=|\beta_{(d+1)1}| \text{ for all }  i\in[d]. 	\end{equation}  
	Consider,
	\begin{align*}
		0&=\norm{\phi(x_{n-3})}-\norm{\phi(x_{n-2})}\\&=|\beta_{dp}|+|\beta_{(d+1)1}|+|\beta_{(d+1)2}|+\sum_{j=1}^{p-3}|\beta_{1j}|-\Big(|\beta_{(d+1)1}|+|\beta_{(d+1)2}|+\sum_{j=1}^{p-3}|\beta_{1j}|+|\beta_{1(p-2)}|\Big)\\&=|\beta_{dp}|-|\beta_{1(p-2)}|. \implies |\beta_{1(p-2)}|=|\beta_{dp}|
	\end{align*} Furthermore, using Equation (\ref{pointos3}), $ |\beta_{1(p-2)}|=|\beta_{dp}|=|\beta_{(d+1)2}|. $ Similarly using $\norm{\phi({x_{n-4})}}-\norm{\phi(x_{n-3})}=0$ we get $ |\beta_{1(p-3)}|=|\beta_{d(p-1)}|=|\beta_{(d+1)1}|. $ Once again using Observation \ref{obsgcd}, $ \sum_{j=1}^{2}|\beta_{(d+1)j}|=2. $ Therefore, we have there possibilities: $ |\beta_{(d+1)1}|=|\beta_{(d+1)2}|=1$ or  $ |\beta_{(d+1)1}|=2 $ and $|\beta_{(d+1)2}|=0 $ or $ |\beta_{(d+1)1}|=0 $ and $|\beta_{(d+1)2}|=2.$ Accordingly, we get the following two cases. 
	\begin{caseof}
		\casea{$ |\beta_{(d+1)1}|=|\beta_{(d+1)2}|=1 $}{In this case $ |\beta_{1(p-2)}|=|\beta_{(d+1)2}|=1 $ and $  |\beta_{1(p-3)}|=|\beta_{(d+1)1}|=1.$  
			On extending we get $ |\beta_{1(p-j)}|=1 $ for all $ j\in[p]. $ This implies $|\beta_{1j}|=1$ for all $j\in[p].$ Therefore by observation \ref{obsgcd} for all $ i\in[d] $ and $ j\in[p] $ we get $ |\beta_{ij}|=1. $ Moreover, as $ |\beta_{(d+1)1}|=|\beta_{(d+1)2}|=1 $  we have all $|\beta_{ij}|$'s as one.  Since we obtained the $\beta_{ij}$'s after re-labeling the vectors $a_k$'s, thus automatically $|a_k|=1$ for all $k\in[n].$ This implies $ \phi $ is a BPM and that is a contradiction as we have chosen $ \phi $ to be a non-BPM and non-UM . Hence this case cannot occur.}
		\casea{$ |\beta_{(d+1)1}|=2,|\beta_{(d+1)2}|=0 $ or $ |\beta_{(d+1)1}|=0,|\beta_{(d+1)2}|=2.$}{We will work with $ |\beta_{(d+1)1}|=2,|\beta_{(d+1)2}|=0 $ and the other case is similar to this. In this case we get $ |\beta_{1(p-2)}|=|\beta_{(d+1)2}|= 0$ and $  |\beta_{1(p-3)}|=|\beta_{(d+1)1}|=2.$ 
			On extending we get  $ |\beta_{1(p-j)}|\in\{0,2\} $ for all $ j\in[p]. $ This implies $|\beta_{1j}|\in\{0,2\}$ for all $j\in[p],$ and $ p=\norm{\phi(x_1)}=\sum_{j=1}^{p} |\beta_{1j}| =2k$ for some $ k. $ This implies $ 2|p $ and in turn $ 2|\GCD(p,n). $ Therefore we get $ \GCD(p,n)\geq 2 $ which is a contradiction. Hence this case cannot occur either. }
	\end{caseof}
\no	Thus $ \phi$ cannot be in $ \nrh{\C}. $ By Theorem \ref{thnrhbpm} we already know the count of BPM that are in $ \nrh{\C} $ to be $ 2n. $ Finally adding the $ n $ unity maps we get the result.
\end{proof}

		Combining the results of Propositions \ref{propgcd1} and \ref{propgcd2} for a circulant code $ \C $ with support $ 2<p<n-1 $ we get that $ |\nrh{\C}|=3n $ for $ n=pd+r $ where $ r\in\{1,2\} $ and $ \GCD(p,n)=1. $ 
Our next aim was to generalize the above Propositions \ref{propgcd1} and \ref{propgcd2} for any $ r$ such that $ n=pd+r$ and $ 0<r<p $.  At this moment we do not have the proof of the generalization, but we  strongly believe in the following conjecture.

\begin{conjecture} \label{conj1}
	Let $ \C $ be a circulant code  with support $ 2<p<n-1 $. If $ \GCD(p,n)=1 $ and $ n=pd+r $ with $ 2<r<p, $ then $ |\nrh{\C}|=3n.$ 
\end{conjecture}
 \begin{remark}\label{remnp3}
 	Note that, if the circulant code with support $ p=3 $ is such that $ \GCD(n,3)=1 $ then $ n=3d+1 $ or $ n=3d+2 $ for some suitable choice of $ d. $ So, if $ n>4, $  Propositions \ref{propgcd1} and \ref{propgcd2} gives us that $ |\nrh{\C}| =3n.$ Moreover, by Proposition \ref{propcirnrh} if $ n=4 $ as $ p=3=n-1 $ we get $ |\nrh{\C}|=n!+n=28 $. Note that when $ p=3 $ we are now only left with $ n=3d $ case. By brute force we counted that $ |\nrh{\C}|=270 $ where $ \C $ is a circulant code on $ n=6 $ neurons with support $ p=3. $ In the next theorem we will work with $ n=3d $ where $ d>2. $
 \end{remark}
\begin{theorem}
	Let $ \C $ be a circulant code on $ n$ neurons with support $ p=3. $ If $ n=3d, $ where $ d>2 $ then $ |\nrh{\C}|=15n+3^2\left( \dfrac{n}{3} \right)!.  $ \label{thnp3k}
\end{theorem}

\begin{proof}
		Let us first count the total number of non-BPM and non-UM  that are in $ \nrh{\C}. $ Let $ \phi $ be a non-BPM and non-UM  with $ \{a_i\}_{i\in[n]} $ as its representing vectors. As observed in the proof of Theorem \ref{thnpk2} we have cases in which  there exists $ i\in[3] $ such that $ |a_i|=3 $ and for all $ j\in[3]\backslash\{i\}, |a_j|=0. $  Also, as there is another partition of $ 3 $ which is not all ones (namely $ 3=2+1 $)  we get more cases which will corresponds to $ |a_1|=2,|a_2|=1,|a_3|=0 $ and its possible permutations. Thus in total we will have these 2 broader class of cases. Let us fix one type of choice and count how many such neural ring endomorphisms it corresponds to. By the choice of all $ |a_i| $ we see that for all $i\in [n],\ \norm{\phi(x_i)}=3. $ This implies for all $ i\in [n] $ there exists $ j\in[n] $ such that $ \phi(x_i^{})=x_j^{}. $
		\begin{caseof}
			\casea{$ (|a_1|,|a_2|,|a_3|)=(3,0,0) $ or $ (|a_1|,|a_2|,|a_3|)= (0,3,0) $ or $ (|a_1|,|a_2|, |a_3|)=(0,0,3). $} Let us consider the sub-case when  $  (|a_1|,|a_2|,|a_3|)=(3,0,0). $ \\ This case is similar to case 1 as in the proof of Theorem \ref{thnpk2}. Firstly it is clear that $ \phi(\rho_1) $ has $ n $ choices and $ \phi(\rho_2)=\phi(\rho_3)=0. $ Next, for $ \phi(\rho_4) $ we have to choose from all the triplets that are left. So we get $ \left( \frac{n}{3}-1\right)  $ choices. Further completing the process we get the total maps that are in $ \nrh{\C} $ as $ n\times\left( \frac{n}{3}-1\right)\times\left( \frac{n}{3}-2\right)\times1=3\left(\frac{n}{3}\right)! $. \\
			 The other 2 subcases will be similar to the above case. Hence Case 1 gives us $ 3^2\left(\frac{n}{3}\right)! $ non-BPM and non-UM maps that are in $ \nrh{\C}. $
			\casea{ For some $ i,j\in[3], i\not =j $  let $ |a_i|=2 $ and $ |a_j |=1$}{Then by permuting $ i,j\in[3] $ we get 6 sub-cases. Consider the sub-case when $ (|a_1|,|a_2|,|a_3|)=(2,1,0) $. } \\ In this sub-case firstly we get that $ \phi(\rho_1) $ can take any consecutive sum of basis elements and so it has $ n $ choices. Let $ \phi(\rho_1)=\rho_l+\rho_{\mdsum{l+1}}. $ Next as $ \phi(x_1)\in\{x_k\} $ it ensures that $ \phi(\rho_2) $ can either be $ \rho_{\mdsum{l+n-1}} $ or $ \rho_{\mdsum{l+2}}. $ We already know that $ \phi(\rho_3)=0. $ Further we observe that this process fixes a unique choice for remaining $ \phi(\rho_k) $ for $ k\in[n]\backslash[3]. $ Hence this sub-case gives us $ 2n $ non-BPM and non-UM  that are in $ \nrh{\C}. $ 
			
			\no The remaining 5 sub-cases will be similar to the above sub-case. Hence we get $  12n$  non-BPM and non-UM  that are in $ \nrh{\C}. $ 
		\end{caseof}
\no	As described in the previous proofs we get $ 3n $ BPM and UM maps that are in $ \nrh{\C}. $ Hence the result.
\end{proof}
\no Looking at the pattern from Theorems  \ref{thnpk2} and \ref{thnp3k} we end our paper with the following conjecture.
\begin{conjecture} \label{conj2}
	Let $ \C $ be a circulant code  with support $ p. $ If \ $ 3<p<n-1$ and $ p|n $ then $ |\nrh{\C}|=(p^2+p+3)n+p^2\left( \dfrac{n}{p}\right)!.  $
\end{conjecture}